 \documentclass[12pt, oneside, titlepage]{amsart}
\usepackage{amsaddr}
\usepackage[left=3cm,right=2cm,top=3cm,bottom=2cm]{geometry}
\usepackage{graphicx,color}
\usepackage{multirow}
\usepackage{amsmath}
\usepackage[english]{babel}
\usepackage{amssymb}
\usepackage[pagewise]{lineno}
\usepackage{epsfig}
\usepackage{graphicx,color}%\textcolor{red}{texto em vermelho}
\usepackage[latin1]{inputenc}%acentos
\usepackage{latexsym}%\Box
\usepackage[normalem]{ulem}%sublinhados
\usepackage{graphpap}
\usepackage{bm}

\usepackage[shortlabels]{enumitem}%para mudar enumerate
\usepackage[pdftex]{hyperref}%links
\usepackage[pdftex]{hyperref}

\DeclareMathOperator{\rank}{\mathrm{rank}}

\newtheorem{teo}{Theorem}[section]

\newtheorem{defi}{Definition}[section]
\newtheorem{obs}{Remark}[section]
\newtheorem{lema}{Lemma}[section]
\newtheorem{prop}{Proposition}[section]
\newtheorem{ex}{Example}[section]

\newcommand{\isEquivTo}[1]{\underset{#1}{\sim}}
\newcommand{\bnu}{\bm{\nu}}

\newcommand{\cD}{\mathcal{D}}
\newcommand{\bD}{\mathbf{D}}

\newcommand{\cT}{\mathcal{T}}

\newcommand{\bxi}{\bm{\xi}}

\newcommand{\by}{\mathbf{y}}
\newcommand{\bx}{\mathbf{x}}
\newcommand{\bn}{\mathbf{n}}

\newcommand{\w}{\mathbf{w}}
\newcommand{\vv}{\mathbf{v}}

\newcommand{\bI}{\mathbf{I}}
\newcommand{\bII}{\mathbf{II}}
\newcommand{\bv}{\mathbf{v}}
\newcommand{\bomega}{\mathbf{\Omega}}
\newcommand{\bLambda}{\mathbf{\Lambda}}

\newcommand{\cal}{\mathcal}
\newcommand{\R}{\mathbb{R}}

\newcommand{\bW}{\mathbf{W}}

\newcommand{\cR}{\mathcal{R}}

\usepackage{mathrsfs}

\title{Equiaffine Structure on frontals}

\author{Igor Chagas Santos}

\thanks{\text{The author was supported by CAPES Proc. PROEX-11365975/D}}
\thanks{\text{\textit{Present address}: Instituto de Matemática e Estatística, Universidade Federal da Bahia, Rua Br}}
\thanks{de Jeremoabo, 40170-115 Salvador, BA, Brazil}

\address{Instituto de Ciências Matemáticas e de Computação, Universidade de
	São Paulo, Av. Trabalhador São-carlense, 400, São Carlos, SP
	13566-590, Brazil.}
\email{igor.chs34@gmail.com}

\usepackage{microtype} 
\usepackage{colonequals}

\begin{document}

\begin{abstract}
	In this paper, we generalize the idea of equiaffine structure to the case of frontals and we define the Blaschke vector field of a frontal. We also investigate some necessary and sufficient conditions that a frontal needs to satisfy to have a Blaschke vector field and provide some examples. Finally, taking the theory developed here into account we present a fundamental theorem, which is a version for frontals of the fundamental theorem of affine differential geometry.
\end{abstract}
{\let\newpage\relax\maketitle}

\section{Introduction}
The systematic treatment of curves and surfaces in unimodular affine space is a main subject in affine differential geometry. This study began around 1916 and mathematicians like Blaschke, Berwald, Pick and Radon were pioneers in the development of this area (see \cite{nomizuhistory} for historical notes). In more recent decades, some papers are dedicated to the study of the geometry of affine immersions and structures related to affine differential geometry, as line congruences, for instance (see \cite{barajas2020lines}, \cite{brucetari}, \cite{esferas1}, \cite{davis}, \cite{nossoartigo}, \cite{Nomizu1}, \cite{luis1}, \cite{luis2}).

The main goal in classical affine differential geometry is the study of properties of surfaces in $3$-dimensional affine spaces, which are invariant under equiaffine transformations. In this sense, given a regular surface $M$ and $\bxi$ an equiaffine transversal vector field on $M$, we obtain an equiaffine structure on $M$ induced by $\bxi$. By equiaffine structure, we mean a torsion free affine connection, a parallel volume element on $M$ and a shape operator associated to $\bxi$. Taking this into account, an important result is the fundamental theorem of affine differential geometry, which asserts that given some integrability conditions there are a surface $M$ and an equiaffine transversal vector field $\bxi$ with a given equiaffine structure (see section 4.9 in \cite{Simon} or chapter 2, section 8 in \cite{nomizu1994affinen}).
%In this context, we can also define a special equiaffine transversal vector field associated to a non-parabolic regular surface $M$, that is the Blaschke (or affine normal) vector field, which is an invariant under equiaffine transformations.%

When studying a regular non-parabolic surface $M$ from the affine viewpoint, an important object is the associated Blaschke (or affine normal) vector field, which is an invariant under equiaffine transformations. A surface $M$ equipped with the equiaffine structure determined by its Blaschke vector field is called a Blaschke surface. With the structure given by this vector field we define, for instance, proper and improper affine spheres, which are special types of Blaschke surfaces that have been studied in many papers (see \cite{esferas1}, \cite{esferas3}, \cite{esferas2}, \cite{esferas4}, \cite{esferas5}, \cite{esferas6}).

The study of surfaces with singularities from the affine differential geometry viewpoint has not been much explored, mainly due to the difficulties which arise at singular points. As we want to explore this viewpoint, in this paper we work with a special class of singular surfaces called frontals. If we take a surface $S$ and we think of light as particles which propagate at unit speed in the direction of the normals of $S$, then at a given time $t$, this particles provide a new surface $S^\prime$. We call $S^\prime$ the wave front of $S$. The notion of frontals arises as a generalization of wave fronts, when considering the case of hypersurfaces. In
recent years, many papers are dedicated to the study of these singular surfaces (see \cite{fukunaga2019framed}, \cite{ishikawa2018singularities}, \cite{ishikawa2}, \cite{luciana1}, \cite{alexandro2}, \cite{alexandro2019fundamental}, \cite{tito3}, \cite{saji1}, \cite{saji2009geometry}). Other references can be found in the survey paper \cite{ishikawa2018singularities}.

Our goal is to extend the study of properties invariant under equiaffine transformations to the case of frontals. We first define the notion of equiaffine transversal vector field to a frontal, using the limiting tangent planes. With this, we provide a definition of affine non-degenerate frontal and define the Blaschke vector field of such a frontal as the smooth extension of the Blaschke vector field defined on its regular part. In theorem \ref{teoremacaracterizacaoBlaschke} some necessary and sufficient conditions that a frontal needs to satisfy to have a Blaschke vector field are shown.

In remark \ref{classesex} we provide some classes of frontals that admit a Blaschke vector field. When studying the class of wave fronts of rank $1$ with extendable Gaussian curvature it turns out that from this class we get a subclass of \textit{frontal improper affine spheres}, i.e. frontals with constant Blaschke vector field. This class seems to be related to improper affine maps, that is a class of improper affine spheres with singularities introduced in \cite{Martinez} for convex surfaces. It is worth observing that improper affine spheres with singularities is a topic of interest in differential geometry, see \cite{esferas1}, \cite{ishikawaimproper}, \cite{milan2013} and \cite{Nakajo}, for instance.

Finally, taking into account the equiaffine theory developed here, we obtain in theorem \ref{teofundamental} a version for
frontals of the fundamental theorem of affine differential geometry for regular surfaces, in a
way that its proof relies on assuming the integrability conditions in the regular case. To prove this theorem the same approach used in \cite{alexandro2019fundamental} is applied, but here we are working not only with the unit normal vector field, but with any equiaffine vector field transversal to a frontal.

This paper is organized as follows. In section \ref{sec2} some well known results from the equiaffine theory for regular surfaces are shown. In section \ref{sec3} we review some content about frontals and investigate some classes which play an important role in the next sections. In section \ref{sec4}, we generalize the idea of equiaffine structure on frontals. Since we have the notion of equiaffine transversal vector field to a frontal, in section \ref{sec5} we define the Blaschke vector field of a frontal and we characterize frontals which have Blaschke vector field. Finally, in section \ref{sec6} we provide a fundamental theorem for the theory developed in the previous sections.

\subsubsection*{Acknowledgements}
This work is part of author's Ph.D thesis, supervised by Maria Aparecida Soares Ruas and Débora Lopes, whom the author thanks for all the support and constant motivation. The author thanks Tito Medina for his support and also for being constantly available for helpful discussions. The author is also grateful to Raúl Oset and the Singularity Group at the Universitat de València for their hospitality and useful comments on this work during author's stay there.
\subsection*{Statements and Declarations}
\subsubsection*{Availability of data}
Data sharing not applicable to this article as no datasets were generated or analyzed during the current study.

\subsubsection*{Conflict of interest} The author has no conflicts of interest to declare that are relevant to the
content of this study

\section{Fixing notations, definitions and some basic results}\label{sec2}
We denote by $U$ an open subset of $\R^2$, where $u = (u_{1}, u_{2}) \in U$ and for a given smooth map $f: U \to \R^n$ the map $Df: U \to M_{n\times 2}(\R)$ is the differential of $f$. 
\subsection{Equiaffine structure for non-parabolic regular surfaces}
Let $\R^3$ be a three-dimensional affine space with volume element given by $\omega(\w_{1}, \w_{2}, \w_{3}) = \det(\w_{1}, \w_{2}, \w_{3})$, for $\w_{1}, \w_{2}, \w_{3} \in \R^3$. Let $D$ be the standard flat connection in $\R^3$ and $\bx: U \rightarrow \R^3$ a regular surface with $\bx(U) = M$ and $\bxi: U \rightarrow \R^3 \setminus \lbrace \bm{0} \rbrace$ a vector field which is transversal to $M$. Then,
\begin{align*}
T_{p}\R^3 = T_{p}M \oplus \langle \bxi(u) \rangle_{\R},
\end{align*}
where $\bx(u) = p$, for any $u \in U$. If $X$ and $Y$ are vector fields on $M$, then we have the decomposition
\begin{align}\label{affinefundam}
D_{X}Y = \nabla_{X}Y + \mathbf{c}(X,Y)\bxi,
\end{align}
where $\nabla$ is the \textit{induced affine connection} and $\mathbf{c}$ is the \textit{affine fundamental form} induced by $\bxi$. We say that $M$ is \textit{non-degenerate} if $\mathbf{c}$ is non-degenerate which is equivalent to say that $M$ is a non-parabolic surface (see chapter 3 in \cite{nomizu1994affine}). Furthermore, we have
\begin{align*}
D_{X}\bxi = -S(X) + \tau(X)\bxi,
\end{align*}
where $S$ is the \textit{shape operator} and $\tau$ is the \textit{transversal connection form}. We say that $\bxi$ is an \textit{equiaffine transversal vector field} if $\tau =0$. The \textit{induced volume element} $\theta $ is defined as follows
\begin{align*}
\theta(X, Y) := \omega(X,Y, \bxi),
\end{align*}
where $X$ and $Y$ are tangent to $M$. 

\begin{defi}\normalfont\label{dettheta}
	Let $\bxi$ be an arbitrary vector field which is transversal to $M$, $\mathbf{c}$ the affine fundamental form and $\theta$ the induced volume element. We define $\det_{\theta}\mathbf{c}$ as $\det(\mathbf{c}_{ij})$, where $\mathbf{c}_{ij} = \mathbf{c}(X_{i}, X_{j})$ and $\lbrace X_{1}, X_{2} \rbrace$ is a unimodular basis for $\theta$, that is, $\theta(X_{1}, X_{2}) = 1$.  
\end{defi}

\begin{obs}\normalfont
	Since the determinant of $\begin{pmatrix}
		\mathbf{c}_{ij}
	\end{pmatrix}$ is independent of the choice of unimodular basis $\lbrace X_{1}, X_{2} \rbrace$,  $\det_{\theta}\mathbf{c}$ is well defined.
\end{obs}

The next proposition relates the induced volume element $\theta$ and the definition of equiaffine transversal vector field.
\begin{prop}{\rm(\cite{nomizu1994affine}, Proposition 1.4)}\label{prop1.4}
	We have
	\begin{align}
	\nabla_{X}\theta = \tau(X)\theta,\; \text{for all $X \in T_{p}M$}.
	\end{align}	
	Consequently, the following two conditions are equivalent:
	\begin{enumerate}
		\item [(a)]$\nabla\theta = 0$.
		\item [(b)] $\tau = 0$.
	\end{enumerate}
\end{prop}

We say that $M$ has a \textit{parallel volume element} if there is a volume element $\theta$ on $M$ such that $\nabla\theta=0$, where
\begin{align*}
\nabla_{X}\theta(X_{1}, X_{2}) = X(\omega(X_{1}, X_{2})) - \theta(X_{1}, \nabla_{X}X_{2}) - \theta(\nabla_{X}X_{1}, X_{2})
\end{align*}
for $X, X_{1}, X_{2}$ vector fields on $M$. Then, it follows from proposition \ref{prop1.4} that a vector field $\bxi$, transversal to a non-parabolic surface, is equiaffine if and only if the induced volume element is parallel.

Given a non-parabolic surface $\bx: U \to \R^3$ the affine fundamental form $\mathbf{c}$ is non-degenerate, then it can be treated as a non-degenerate metric (not necessarily positive-definite) on $\bx(U) = M$.
\begin{defi}\label{defBlaschkeusual}\normalfont
	Let $\bx: U \rightarrow \R^3$ be a non-parabolic surface. A transversal vector field $\bxi: U \rightarrow \R^3 \setminus \lbrace \bm{0} \rbrace$ is \textit{the Blaschke normal vector field} of $\bx(U) = M$ if the following conditions hold:
	\begin{enumerate}
		\item [(a)] $\bxi$ is equiaffine.
		\item [(b)] The volume element $\theta$ induced by $\bxi$ coincides with the volume element $\omega_{\mathbf{c}}$ of the non-degenerate metric $\mathbf{c}$.
	\end{enumerate}
	
\end{defi}
From now on, we refer to the Blaschke vector field given in definition \ref{defBlaschkeusual} as the \textit{usual Blaschke vector field}.

\section{Frontals}\label{sec3}
 A smooth map $\bx: U \to \R^3$ is said to be a \textit{frontal} if, for all $q \in U$ there is a vector field $\bn: U_{q} \to \R^3$ where $q \in U_{q}$ is an open subset of $U$, such that $\lVert \bn \rVert = 1$ and $\langle \bx_{u_i}(u), \bn(u) \rangle = 0$, for all $u \in U_{q}$, $i=1,2$. This vector field is said to be a unit normal vector field along $\bx$. We say that a frontal $\bx$ is a \textit{wave front} if the map $(\bx, \bn): U \to \R^3 \times \mathbb{S}^2$ is an immersion for all $q \in U$. Here, we consider mainly \textit{proper frontals}, that is, frontals $\bx$ for which the singular set $\Sigma(\bx) = \lbrace q \in U: \text{$\bx$ is not immersive at $q$} \rbrace$ has empty interior. This is equivalent to say that $U \setminus \Sigma(\bx)$ is an open dense set in $U$.

\begin{defi} \label{def21}\normalfont
	We call a \textit{moving basis} a smooth map $\bomega: U \rightarrow \cal{M}_{3\times 2}(\R)$ in which the columns $\w_{1}, \w_{2}: U \rightarrow \R^3$ of the matrix $\bomega =  \begin{pmatrix}
	\w_{1} & \w_{2}
	\end{pmatrix} $ are linearly independent vector fields.
\end{defi}

\begin{defi} \label{def22}\normalfont
	We call a \textit{tangent moving basis} (tmb) of $\bx$ a moving basis $\bomega = \left( \w_{1}, \w_{2} \right)$, such that $\bx_{u_{1}}, \bx_{u_{2}} \in \langle \w_{1}, \w_{2} \rangle_{\R}$, where $\langle \;, \rangle_{\R}$ denotes the linear span $\R$-vector space.
\end{defi}
Given a tmb $\bomega = \begin{pmatrix}
\w_{1} & \w_{2}
\end{pmatrix}$ we denote by $T_{\Omega}(q)$ the plane generated by $\w_{1}(q)$ and $\w_{2}(q)$, for all $q \in U$. Note that given two tangent moving basis $\bomega$ and $\widetilde{\bomega}$ of a proper frontal, we get $T_{\Omega}  = T_{\widetilde{\Omega}}$. The next proposition provides a characterization of frontals in terms of tangent moving basis.
\begin{prop} \rm{(\cite{alexandro2019fundamental}, Proposition 3.2)}\label{propdecomp}
	Let $\bx: U \rightarrow \R^3$ be a smooth map with $U \subset \R^2$ an open set. Then,
	$\bx$ is a frontal if and only if, for all $q \in U$, there are smooth maps  $\bomega: U_{q} \rightarrow \cal{M}_{3\times 2}(\R)$ and $\bLambda: U_{q} \rightarrow \cal{M}_{2\times 2}(\R)$ with $\text{rank}(\bomega)= 2$ and $U_{q} \subset U$ an open neighborhood of $q$, such
	that $D\bx(\tilde{q}) = \bomega \bLambda_{\Omega}^T$, for all $\tilde{q} \in U_{q}$.
\end{prop}

\begin{obs}\normalfont
	It follows from \ref{propdecomp} that $\bxi: U \to \R^3$ is a frontal if, and only if, there is a tangent moving basis of $\bx$.
\end{obs}

Since a tangent moving basis exist locally and we want to describe local properties, from now on we suppose that for a given frontal we have a global tangent moving basis. Then, if a frontal $\bx$ satisfies $D\bx = \bomega \bLambda^{T}$, where $\bomega$ is a tangent moving basis, we have that $\Sigma(\bx) = \lambda_{\Omega}^{-1}(0)$, where $\lambda_{\Omega} := \det \bLambda$.

Let $\bx: U \rightarrow \R^3$ be a frontal, $\bomega = \begin{pmatrix}
\w_{1} & \w_{2}
\end{pmatrix}$ a tmb of $\bx$ and denote by $\bn = \dfrac{\w_{1}\times \w_{2}}{\lVert \w_{1} \times \w_{2} \rVert}$ the unit normal vector field induced by $\bomega$. We set the matrices
\begin{align}
\bI_{\Omega} &:= \bomega^{T}\bomega = \begin{pmatrix}
E_{\Omega} & F_{\Omega}\\
F_{\Omega} & G_{\Omega}
\end{pmatrix} = \begin{pmatrix}
\langle \w_{1}, \w_{1} \rangle & \langle \w_{1}, \w_{2} \rangle\\
\langle \w_{2}, \w_{1} \rangle & \langle \w_{2}, \w_{2} \rangle
\end{pmatrix}\nonumber,\\
\bII_{\Omega} &:= -\bomega^{T}D\bn = \begin{pmatrix}
e_{\Omega} & f_{1\Omega}\\
f_{2\Omega} & g_{\Omega}
\end{pmatrix} = \begin{pmatrix}
-\langle \w_{1}, \bn_{u_1} \rangle & -\langle \w_{1}, \bn_{u_2} \rangle\\
-\langle \w_{2}, \bn_{u_1} \rangle & -\langle \w_{2}, \bn_{u_2} \rangle
\end{pmatrix}\nonumber,\\
\bm{\mu}_{\Omega} &:= -\bII_{\Omega}^{T}\bI_{\Omega}^{-1}\nonumber,\\
\cT_{1}&:=(\bomega_{u_1}^T\bomega)\bI_\Omega^{-1}=\begin{pmatrix}
\cT_{11}^{1} & \cT_{11}^{2}\\
\cT_{21}^{1} & \cT_{21}^{1}
\end{pmatrix},\label{T1O} 
\\
\cT_{2}&:=(\bomega_{u_2}^T\bomega)\bI_\Omega^{-1}=\begin{pmatrix}
\cT_{12}^{1} & \cT_{12}^{2}\\
\cT_{22}^{1} & \cT_{22}^{1}
\end{pmatrix}\label{T2O}.
\end{align}
Given a frontal $\bx: U \to \R^3$ and a tmb $\bomega = \begin{pmatrix}
\w_{1} & \w_{2}
\end{pmatrix}$ of $\bx$, it follows that  $\langle \w_{1}, \bn \rangle = \langle \w_{2}, \bn \rangle = 0$. By taking partial derivatives of these equalities, we rewrite
\begin{align}\label{matrizII}
\mathbf{II}_{\Omega} = \begin{pmatrix}
\langle (\w_{1})_{u_{1}}, \bn \rangle & \langle (\w_{1})_{u_{2}}, \bn \rangle \\ 
\langle (\w_{2})_{u_{1}}, \bn \rangle & \langle (\w_{2})_{u_{2}}, \bn \rangle
\end{pmatrix}. \end{align}
\begin{defi} \label{curvaturerelativa}\normalfont
	Let $\bx: U \rightarrow \R^3$ be a frontal and $\bomega$ a tangent moving basis of $\bx$. We define the $\bomega$-\textit{relative curvature} $K_{\Omega}:= \det(\bm{\mu_{\Omega}})$.
\end{defi}

Given a frontal $\bx: U \to \R^3$ with a global unit normal vector field $\bn: U \to \R^3$, we can also consider the matrices
\begin{align}\label{segform}
\bI &:= D\bx^{T}D\bx = \begin{pmatrix}
E & F\\
F & G
\end{pmatrix} = \begin{pmatrix}
\langle \bx_{u_1}, \bx_{u_1} \rangle & \langle \bx_{u_1}, \bx_{u_2} \rangle\\
\langle \bx_{u_1}, \bx_{u_1} \rangle & \langle \bx_{u_2}, \bx_{u_2} \rangle
\end{pmatrix}\nonumber,\\
\bII &:= -D\bx^{T}D\bn = \begin{pmatrix}
e & f\\
f & g
\end{pmatrix} = \begin{pmatrix}
-\langle \bx_{u_1}, \bn_{u_1} \rangle & -\langle \bx_{u_1}, \bn_{u_2} \rangle\\
-\langle \bx_{u_2}, \bn_{u_1} \rangle & -\langle \bx_{u_2}, \bn_{u_2} \rangle
\end{pmatrix}.
\end{align}
If we decompose $D\bx = \bomega \bLambda^{T}$, then $\bI = \bLambda \bI_{\Omega} \bLambda^{T}$ and $\bII = \bLambda\bII_{\Omega}$. Also, the classical normal curvature at a regular point $q \in U$ is given by
\begin{align*}
k_{q}(\vartheta) := \dfrac{\vartheta^{T} \bII \vartheta}{\vartheta^{T}\bI \vartheta},
\end{align*}
where $\vartheta \in \R^2 \setminus \lbrace 0 \rbrace$ are the coordinates of a vector in the basis $\begin{pmatrix}
\bx_{u_1} & \bx_{u_2}
\end{pmatrix}$.

Let $\bx: U \rightarrow \R^3$ be a frontal, $\bomega: U \rightarrow M_{3 \times 2}(\R)$ a tmb of $\bx$, where $\bomega = \begin{pmatrix}
	\w_{1} & \w_{2}
\end{pmatrix}$ and $\bn: U \rightarrow \R^3$ the unit normal vector field along $\bx$. For each $q \in U$ we decompose
\begin{align*}
	\R^3 = T_{\Omega}(q) \oplus \langle \bn(q) \rangle_{\R}.
\end{align*}
Using this decomposition we get
\begin{align}\label{decompo2}
	\w_{i_{u_{j}}} =\cT^{1}_{ij} \w_{1} + \cT^{2}_{ij}  \w_{2} + p_{ij}\bn,
\end{align}
where the symbols $\cT_{ij}^{k}$, $i,j,k=1,2$ are those in (\ref{T1O}) and (\ref{T2O}). Note that $p_{ij} = \langle (\w_{i})_{u_{j}}, \bn \rangle$, thus the matrix $\begin{pmatrix}
	p_{ij}
\end{pmatrix}$ coincide with the matrix (\ref{matrizII}).
\begin{obs}\normalfont \label{naodegen}
	If we define a bilinear form $p_{\Omega}(q): T_{\Omega} \times T_{\Omega} \rightarrow \R$, given by $p_{\Omega}(q)({\w_{i},\w_{j}}) = p_{ij} = \langle \w_{i_{u_{j}}}(q), \bn(q) \rangle $, then the matrix of $p_{\Omega}$ relative to the basis $\bomega$ is $\mathbf{II}_{\Omega}$ and  $p_{\Omega}$ is non-degenerate if and only if $\mathbf{II}_{\Omega}$ is non-singular, which is equivalent to say that the $\bomega$-relative curvature $K_{\Omega}$ is non-zero (see definition \ref{curvaturerelativa}).
\end{obs}
\begin{defi}\label{deffronntalnp}\normalfont
	A proper frontal is said to be a \textit{non-parabolic frontal} if for some \text{tmb} $\bomega$ the relative curvature $K_{\Omega}$ never vanishes.
\end{defi}	

\begin{obs}\normalfont
	It follows from corollary 3.23 in \cite{alexandro2019fundamental} that a frontal $\bx: U \rightarrow \R^3$ is a wavefront if and only if, $(K_{\Omega}, H_{\Omega}) \neq \mathbf{0}$ on $\Sigma(\bx)$, for whatever tangent moving base $\bomega$.  Therefore, every non-parabolic frontal is a wavefront.
\end{obs}

\begin{prop}\label{propnaopa}
	The notion of non-parabolicity is independent of \text{tmb}.
\end{prop}
\begin{proof}
	Via proposition 3.18 in \cite{alexandro2019fundamental} the zeros of $K_{\Omega}$ are independent of \text{tmb}.
\end{proof}

The next proposition provides a representation formula for non-parabolic frontals.
\begin{prop}
	Let $\bx: (U, 0) \rightarrow (\R^3,0)$ be a germ of non-parabolic frontal, $\bomega$ a tmb of $\bx$ and $0 \in \Sigma(\bx)$. Then, up to an isometry $\bx$ is $\cR$-equivalent to 
	\begin{align*}
		\by(u_{1}, u_{2}) = (a(u_{1}, u_{2}),b(u_{1}, u_{2}), \int_{0}^{u_{1}}(t_{1}a_{u_1}(t_{1},u_{2}) + u_{1}b_{u_1}(t_{1},u_{2}))dt_{1} + \int_{0}^{u_{2}}t_{2}b_{u_2}(0,t_{2})dt_{2}),
	\end{align*}
	where $a, b$ are smooth functions such that $a_{u_{2}} = b_{u_1}$.
\end{prop}

\begin{proof}
	See proposition 4.1 in \cite{tito3}.
\end{proof}

\subsection{Frontals with extendable Gaussian curvature}
Now, taking into account the results in \cite{tito3}, we investigate some classes of frontals for which the Gaussian curvature admits a smooth extension. These classes play an important role in section \ref{sec5}, where we define the Blaschke vector field of a frontal.

\subsubsection{Frontals with extendable normal curvature}\label{firstclass}\hfill\\
Let $\bx: U \to \R^3$ be a proper frontal with extendable normal curvature. It follows from corollary 3.1 in \cite{tito3} that if $\bx: (U, 0) \to (\R^3, 0)$ is a frontal with extendable normal curvature and $0$ a singularity of rank $1$, then its Gaussian curvature has a smooth extension. Furthermore, this extension is non-vanishing if and only if $\bx \isEquivTo{} \bn$, where $\bn$ is the unit normal vector field of $\bx$ and $\isEquivTo{}$ indicates that there are $\bomega_{1}$ and $\bomega_{2}$ tmb of $\bx$ and $\bn$, respectively and a smooth matrix valued map $\mathbf{B}: U \to GL(2, \R)$, such that $\bLambda_{2} = \bLambda_{1} \mathbf{B}$, where $D\bx = \bomega_{1} \bLambda_{1}^{T}$ and $D\bn = \bomega_{2} \bLambda_{2}^{T}$. The next theorem characterizes proper frontals of rank $1$ with extendable normal curvature.

\begin{teo}{\rm{({\cite{tito3}, Theorem 3.2})}}
Let $\bx:(U,0) \to (\R^3,0)$ be a proper frontal with extendable normal curvature and $0$ a singularity of rank 1, then after a rigid motion and a change of coordinates on a neighborhood of $0$, $\bx$ can be represented by the formula:
\begin{align*}
\by = (u_{1},b(u_{1},u_{2}),
&\int_{0}^{u_{2}}\left(\int_{0}^{t_2}h(u_{1},t_1)b_{u_2}(u_{1},t_1)dt_1\right)b_{u_2}(u_{1},t_2)dt_2+\int_{0}^{u_{2}}\left(\int_{0}^{u_{1}}l(t_1)dt_1\right)b_{u_2}(u_{1},t_2)dt_2\\
&+\int_{0}^{u_{1}}\left(\int_{0}^{t_2}l(t_1)dt_1\right)b_{u_1}(t_2,0)dt_2+\int_{0}^{u_{1}}\left(\int_{0}^{t_2}r(t_1)dt_1\right)dt_2),\nonumber
\end{align*}
where $b,h,l,r$ are smooth function on  neighborhoods of the origin in each case.
\end{teo}

\subsubsection{Wave fronts with extendable Gaussian curvature}\label{secondclass}\hfill\\
If we look at a germ of wave front $\bx: (U,0) \to (\R^3,0)$, such that $0 \in \Sigma(\bx)$ and $\rank D\bx(0) = 1$, then it follows from remark 4.1 in \cite{tito3} that up to an isometry, $\bx$ is $\cR$-equivalent to 
\begin{align}\label{formwave1}
\by = (u_{1}, -h_{u_{2}}(u_1,u_2), \int_{0}^{u_{1}}(h_{u_{1}}(t, u_{2}) - u_{2}h_{u_{2}u_{1}}(t,u_{2}))dt - \int_{0}^{u_{2}}th_{u_{2}u_{2}}(0,t)dt).
\end{align}
By $\cR$-equivalence we mean that there is a germ of diffeomorphism $\tilde{h}: (\tilde{U}, 0) \to (U, 0)$ such that $\by = \bx \circ \tilde{h}$. Note that $D\by$ has decomposition $D\by = \bomega \bLambda^{T}$, where
\begin{align*}
\bomega = \begin{pmatrix}
1 & 0\\
0 & 1\\
g_{1} & u_{2}
\end{pmatrix}\; \text{and}\; \bLambda^{T} = \begin{pmatrix}
1 & 0\\
b_{u_1} & b_{u_2}
\end{pmatrix},
\end{align*}
for $h$, $g_{1}$ and $b$ smooth functions such that $g_{1} = h_{u_{1}}$ and $-b = h_{u_{2}}$. From this, we obtain that 
\begin{align}\label{gaussiancurvwave1}
K_{\Omega} = \dfrac{-h_{u_{1}u_{1}}}{(1+h_{u_{1}}^2 + u_{2}^2)^2}\; \text{and}\; \lambda_{\Omega} = -h_{u_2u_2}.
\end{align}
From corollary 4.3 in \cite{tito3}, it follows that if $h$ in (\ref{formwave1}) satisfies the equation $h_{u_{1}u_{1}} + c(u_{1}, u_{2})h_{u_{2}u_{2}} = 0$, where $c(u_{1}, u_{2})$ is a smooth function, then $\by$ is a wave front of rank $1$ with extendable Gaussian curvature. Furthermore, using this information in (\ref{gaussiancurvwave1}) and taking into account that on $U \setminus \Sigma(\bx)$ the Gaussian curvature is $K = \dfrac{K_{\Omega}}{\lambda_{\Omega}}$, we get, by the density of $U \setminus\Sigma(\bx)$, that its extension is given by
\begin{align*}
K = \dfrac{-c(u_{1}, u_{2})}{(1+h_{u_{1}}^2 + u_{2}^2)^2}.
\end{align*}
Thus, if $c(u_{1}, u_{2})$ is a non-vanishing function, we obtain a wave front of rank $1$ for which the Gaussian curvature has a non-vanishing extension. It follows from proposition 3.4 in \cite{tito3} that a wave front does not admit a smooth extension for its normal curvature, hence this class has no intersection with the class \ref{firstclass}.

\begin{obs}\label{classe1}\normalfont
	Note that the second order linear PDE
	\begin{align}\label{pde}
	h_{u_{1}u_{1}} + c(u_{1}, u_{2})h_{u_{2}u_{2}} = 0
	\end{align}
	is an important step in order to obtain frontals with extendable non-vanishing Gaussian curvature. If $c(u_{1}, u_{2}) = 1$, then the equation (\ref{pde}) is the two dimensional \textit{Laplace equation} (an elliptic equation), which was discovered by Euler in 1752. This PDE has been useful in many areas, like gravitational potential, propagation of heat, electricity and magnetism (for more information see chapter 7 in \cite{livropde}). On the other hand, if we take $c(u_{1}, u_{2}) = -a^2$, where $a \neq 0$ is a constant, then (\ref{pde}) is the one-dimensional \textit{wave equation} (a hyperbolic equation). The wave equation governs the dynamics of some physical systems, for instance, the guitar string, the longitudinal vibrations of an elastic bar, propagation of acoustic, fluid, and electromagnetic waves (see chapter 5 in \cite{livropde}). Note in (\ref{gaussiancurvwave1}) that the sign of $K$ is determined by the sign of $c(u_{1}, u_{2})$, but this sign also identifies if the PDE (\ref{pde}) is hyperbolic or elliptic.
	\end{obs}

\section{Equiaffine structure on frontals}\label{sec4}
In this section we seek to generalize the idea of equiaffine structure to the case of frontals. We define equiaffine transversal vector fields to a frontal similarly as defined in \cite{nomizu1994affinen} when considering regular surfaces, but as we are dealing with frontals, we need to take into account tangent moving basis and the limiting tangent planes. 

Let $\bxi: U \rightarrow \R^3$ be a vector field which is transversal to the frontal $\bx: U \rightarrow \R^3$ i.e. $\bxi(q) \notin T_{\Omega}$ for all $q \in U$, where $\bomega = \begin{pmatrix}
\w_{1} & \w_{2}
\end{pmatrix}$ is a tmb. Thus, for each $q \in U$ we can decompose 
\begin{align*}
\R^3 = T_{\Omega}(q) \oplus \langle \bxi(q) \rangle_{\R}.
\end{align*}	
If we write
\begin{align}\label{decompo4}
\w_{i_{u_{j}}} =\cD^{1}_{ij} \w_{1} + \cD^{2}_{ij}  \w_{2} + 
h_{ij}\bxi.
\end{align}
then we obtain a bilinear form $h_{\Omega}(q): T_{\Omega} \times T_{\Omega} \rightarrow \R$, such that $h_{\Omega}(q)(\w_i, \w_j) = h_{ij}(q)$. We call $h_{\Omega}$ the \textit{relative affine fundamental form} of $\bx$ induced by $\bxi$. In a similar way, we write
\begin{align}\label{decompo5}
\bxi_{u_i} = -S_{i}^{1}\w_{1} - S_{i}^{2}\w_{2} +  \tau_{i}\bxi.
\end{align}
Then for each $q \in U$ we have $S_{\Omega}(q): T_{\Omega}(q) \rightarrow T_{\Omega}(q)$, such that $S_{\Omega}(q)(\w_{i}) = S_{i}^{1}\w_{1} +  S_{i}^{2}\w_{2}$ and $\tau_{\Omega}(q): T_{\Omega}(q) \rightarrow \R$, such that $\tau_{\Omega}(q)(\w_{i}) = \tau_{i}$, $i=1,2$. We call $S_{\Omega}$ the \textit{relative shape operator} of $\bxi$ and $\tau_{\Omega}$ the \textit{relative transversal connection form}. 
\begin{defi}\normalfont\label{defequiafim}
The vector $\bxi$ defines an \textit{equiaffine structure} on $\bx$ (or $\bxi$ is \textit{equiaffine}) when the derivatives of $\bxi$ are in $T_{\Omega}(q)$, for all $q \in U$, i.e. when $\tau_{\Omega} \equiv 0$.
\end{defi}		

For a frontal $\bx: U \to \R^3$ and an equiaffine transversal vector field $\bxi: U \to \R^3$ we say that the symbols $\cD_{ij}^{k}$ and $h_{ij}$, given in (\ref{decompo4}), define an equiaffine structure on $\bx$.

\begin{defi}\normalfont
	Given $\bv_{1}, \bv_{2} \in T_{\Omega}$, we define
	\begin{align*}
	\theta(\bv_{1}, \bv_{2}) := \tilde{\omega}(\bv_{1}, \bv_{2}, \bxi),
	\end{align*}
	where $\tilde{\omega}$ is the canonical volume element in $\R^3$. The volume element $\theta$ is called the \textit{induced volume element}.
\end{defi}

\begin{prop}\label{prop1.3}
	Let $\bxi: U \rightarrow \R^3$ be a vector field which is transversal to a frontal $\bx: U \rightarrow \R^3$. Let us suppose that 
	\begin{align*}
	\bxi = \phi\tilde{\bxi} + Z
	\end{align*} 
	where $Z(u) = a(u)\w_{1}(u) + b(u)\w_{2}(u) \in T_{\Omega}$, $\bomega = \begin{pmatrix}
		\w_{1} & \w_{2}
	\end{pmatrix}$, $\tilde{\bxi}$ is an arbitrary transversal vector field and $\phi(u) \neq 0$, for all $u \in U$. Then, 
	\begin{enumerate}[(a)]
		\item \label{itema} $h_{\Omega} = \frac{1}{\phi}\tilde{h}_{\Omega}$.\label{eq1}
		\item \label{itemb} $\tau_{\Omega}(\w_{i}) = \frac{1}{\phi}\left(\tilde{h}_{\Omega}(Z,\w_{i}) + \phi_{u_{i}}\right)$.\label{eq2}
		\item \label{itemc} $\theta = \phi \tilde{\theta}$,
	\end{enumerate}
where $\tilde{h}_{\Omega}$, $\tilde{\theta}$ and $\tilde{\Omega}$ are associated to $\tilde{\bxi}$.

\end{prop}
\begin{proof}
	\begin{enumerate}[(a)]
		\item From $	\bxi = \phi\tilde{\bxi} + Z$, we get
	\begin{align*}
		\frac{\bxi - Z}{\phi} = \tilde{\bxi}.
	\end{align*}
	If we decompose $\w_{iu_{j}} = \tilde{\cD}^{1}_{ij} \w_{1} + \tilde{\cD}^{2}_{ij}  \w_{2} + 
	\tilde{h}_{ij}\tilde{\bxi}$, then
	\begin{align*}
		\w_{iu_{j}} &= \tilde{\cD}^{1}_{ij} \w_{1} + \tilde{\cD}^{2}_{ij}  \w_{2} + 
		\tilde{h}_{ij}\tilde{\bxi}\\
		&= \tilde{\cD}^{1}_{ij} \w_{1} + \tilde{\cD}^{2}_{ij}  \w_{2} + 
		\tilde{h}_{ij} \frac{\bxi - Z}{\phi}\\
		&= \left( \tilde{\cD}^{1}_{ij} \w_{1} + \tilde{\cD}^{2}_{ij}  \w_{2} - \frac{\tilde{h}_{ij}}{\phi}Z \right) + \frac{\tilde{h}_{ij}}{\phi}\bxi\\
		&= \cD^{1}_{ij} \w_{1} + \cD^{2}_{ij}  \w_{2} + 
		h_{ij}\bxi.
	\end{align*}
	Therefore, $h_{ij} = \dfrac{\tilde{h}_{ij}}{\phi}$, $i,j = 1,2$.
\item 	Let us write $B_{\Omega} = \begin{pmatrix}
		b_{ij}
	\end{pmatrix}$ the relative shape operator of $\tilde{\bxi}$. We know that
	\begin{align}\label{1}
		\bxi_{u_{i}} = -S_{\Omega}(\w_{i}) + \tau_{\Omega}(\w_{i})\bxi = -S_{\Omega}(\w_{i}) + \tau_{\Omega}(\w_{i})Z + \tau_{\Omega}(\w_{i})\phi \tilde{\bxi} .
	\end{align}
	On the other hand, 
	\begin{align}\label{2}
		\bxi_{u_{i}}(q) &= -\phi B_{\Omega}(q)(\w_{i}) + \phi_{u_{i}}\tilde{\bxi} + Z_{u_{i}}\\ 
		&= \left(Z_{u_{i}}^{\top} -\phi B_{\Omega}(\w_{i}) \right) + \left(\phi_{u_{i}} + \tilde{h}_{\Omega}(Z, \w_{i})\right)\tilde{\bxi} , \nonumber
	\end{align}
	where $Z_{u_{i}}^{\top}$ is the tangent component of $Z_{u_i}$. By comparing the transversal components of (\ref{1}) and (\ref{2}), it follows that $\tau_{\Omega}(\w_{i}) = \frac{1}{\phi}\left(\tilde{h}_{\Omega}(Z,\w_{i}) + \phi_{u_{i}}\right)$.
	
	\item By definition $\theta(\bv_{1}, \bv_{2}) = \omega(\bv_{1}, \bv_{2}, \bxi) = \omega(\bv_{1}, \bv_{2},  \phi\tilde{\bxi} + Z) = \omega(\bv_{1}, \bv_{2}, \phi\tilde{\bxi}) = \phi \tilde{\theta}(\bv_{1}, \bv_{2})$, where $\bv_{1}, \bv_{2} \in T_{\Omega}$.
\end{enumerate}
\end{proof}

\begin{prop}\label{prop1.2}
If $\bx$ is a non-parabolic frontal then the relative fundamental form induced by a transversal vector field is non-degenerate.
\end{prop}

\begin{proof}
	Let $\bxi$ be a transversal vector field, thus we can write
	\begin{align*}
	\bxi = \phi\tilde{\bxi} + Z,
	\end{align*}
	where $Z$ is tangent, $\tilde{\bxi}$ is an arbitrary transversal vector field and $\phi: U \rightarrow \R \setminus 0$ is smooth. It follows from proposition \ref{prop1.3} that $h_{ij} = \dfrac{\tilde{h}_{ij}}{\phi}$, $i,j = 1,2$. From this, $\begin{pmatrix}
		h_{ij}
	\end{pmatrix}$ is non-singular if and only if $\begin{pmatrix}
		\tilde{h}_{ij}
	\end{pmatrix}$ is non-singular, for $\tilde{h}_{ij}$ associated to an arbitrary transversal vector field. As $\bx$ is non-parabolic, the result follows from remark \ref{naodegen}.
\end{proof}
\begin{obs}\normalfont
	Proposition \ref{prop1.2} shows that the notion of non-parabolicity of a frontal can be verified using any transversal vector field and its relative affine fundamental form. This notion is also independent of tangent moving basis. Thus, non-parabolicity is a property that belongs to the frontal and it is independent of metric.
\end{obs}

\begin{prop}\label{propdet}
	Let $\bx: U \rightarrow \R^3$ be a frontal, $\bxi$ and $\tilde{\bxi}$ arbitrary transversal vector fields with associated affine fundamental forms induced on $U \setminus \Sigma(\bx)$ given by $\mathbf{c}$ and $\tilde{\mathbf{c}}$, respectively. If the induced volume elements are $\theta$ and $\tilde{\theta}$, then $\det_{\theta}\mathbf{c}$ admits a smooth non-vanishing extension to $U$ if and only if $\det_{\tilde{\theta}}\tilde{\mathbf{c}}$ admits a smooth non-vanishing extension to $U$
	
\end{prop}

\begin{proof}
If we write $\bxi = \phi\tilde{\bxi} + Z$ where $Z$ is tangent, it follows analogously to proposition \ref{prop1.3} that $\mathbf{c} = \frac{1}{\phi}\tilde{\mathbf{c}}$ on $U \setminus \Sigma(\bx)$. Since $\theta = \phi \tilde{\theta}$, it follows from well known properties of the determinant funcion that
\begin{align}\label{eqdet}
	\phi^{4}\det_{\tilde{\theta}}\tilde{\mathbf{c}} = \det_{\theta}\mathbf{c}.
\end{align} Then, the result follows from (\ref{eqdet}).
\end{proof}

\begin{defi}\normalfont\label{defaffinenondeg}
	We say that a frontal $\bx: U \to \R^3$ is \textit{affine non-degenerate} if for an arbitrary transversal vector field $\det_{\theta}\mathbf{c}$ admits a smooth non-vanishing  extension to $U$.
\end{defi}

\section{The Blaschke and the conormal vector fields of a frontal}\label{sec5}
The Blaschke vector field and the conormal vector field associated to an equiaffine transversal vector field play an important role when studying regular surfaces from the affine differential geometry viewpoint. With the Blaschke structure, for instance, we define proper and improper affine spheres (see chapter 2 in \cite{nomizu1994affinen}) and the Blaschke line congruences (see section 6 in \cite{nossoartigo}). On the other hand, the conormal vector field makes calculations with the affine support function easier, for instance (see section 1 in \cite{cecil1994focal}).
Taking into account the importance of these two objects, in this section we define the Blaschke vector field of a frontal and the conormal vector field associated to an equiaffine vector field transversal to a frontal.
\subsection{The Blaschke vector field of a frontal}

\begin{defi}\normalfont\label{defBlaschke}
	Let $\bx: U \rightarrow \R^3$ be a proper frontal that is affine non-degenerate. We say that a transversal vector field $\bxi$ is the \textit{Blaschke vector field} of $\bx$ if it is a smooth extension of the usual Blaschke vector field defined on $U \setminus \Sigma(\bx)$.
\end{defi}
It follows from the density of $U \setminus \Sigma(\bx)$ and from the fact that the Blaschke vector field is unique up to sign (see \cite{nomizu1994affinen}) that the above extension is unique. Now, looking at the special affine group (or equiaffine group) \begin{align*}\mathbf{SA}(3,\R) = \lbrace \Phi: x \mapsto \mathbf{A}x + \mathbf{b}: \mathbf{A} \in M_{3}(\R), \det\mathbf{A} = 1\; \text{and}\; \mathbf{b}\; \text{is a constant vector}\rbrace
\end{align*}
we seek to show an invariance property for the Blaschke vector field defined here, in the following sense. Given a frontal $\bx: U\rightarrow \R^3$ and $\Phi \in \mathbf{SA}(3,\R)$, the Blaschke vector field of $\by = \Phi \circ \bx$ is $\overline{\bxi}$, where
\begin{align*}
\overline{\bxi}(q) = \Phi_{*}\bxi(q),\; \text{for all}\; q \in U.
\end{align*}
\begin{prop}
	Let $\bx: U \rightarrow \R^3$ be a proper frontal that is affine non-degenerate. If there exists, the Blaschke vector field of $\bx$ is an equiaffine invariant.
\end{prop}

\begin{proof}
	Let $\Phi(x) = \mathbf{A}x + \mathbf{b}$ be an equiaffine transformation, then $\by = \Phi \circ \bx$ has the same singular set of $\bx$. Furthermore, it is known that on $U \setminus \Sigma(\bx)$ the usual Blaschke vector field of $\by$ is given by $\overline{\bxi} = \Phi_{*} \bxi$, where $\bxi$ is the Blaschke vector field of $\bx$. Let us keep the same notation for the extension of $\bxi$, then $\Phi_{*}\bxi$ is an extension of $\overline{\bxi}$ to $U$, so it is the Blaschke vector field of $\by$. 
\end{proof}	
%\begin{obs}\normalfont
%Note that is not always possible to obtain a smooth extension of the Blaschke vector field. For instance, if $\bx: U\rightarrow \R^3$ is a non-parabolic frontal, i.e., for all tmb $\bomega$, the $\bomega$-relative curvature $K_{\Omega} \neq 0$, then the Gaussian curvature is not extendable (see proposition 4.1 in \cite{alexandro2}). It is known that the usual Blaschke vector field is given on $U \setminus \Sigma(\bx)$ by $\lvert K \rvert^{1/4}\bn + W$, where $W$ is a tangent vector field (see remark \ref{obsBlaschkeusual}), thus $\langle \bxi, \bn \rangle = \lvert K \rvert^{1/4}$. Since it is not possible to extend $K$, it follows that is not possible to extend this vector field. From this, it follows that a necessary condition to obtain the Blaschke vector field of a frontal is that its Gaussian curvature is extendable.
%\end{obs}
Next, taking into account proposition \ref{propdet}, we characterize frontals for which it is possible to define the Blaschke vector field.
\begin{teo}\label{teoremacaracterizacaoBlaschke}
	A frontal $\bx: U \to \R^3$ admits a Blaschke vector field if and only if is affine non-degenerate and there are $a,b \in C^{\infty}(U, \R)$ such that
		\begin{align}\label{exptgblaschkelim0}
 \begin{pmatrix}
 a(q) \\ b(q)
 \end{pmatrix} = \lim\limits_{u \to q}\begin{pmatrix}
 \tilde{h}_{11} & \tilde{h}_{21}\\
 \tilde{h}_{12} & \tilde{h}_{22}
\end{pmatrix}^{-1} \begin{pmatrix}
	-\phi_{u_1}\\
	-\phi_{u_2}
	\end{pmatrix},\; \text{for all $q \in U$},
	\end{align} 
	where $\bomega = \begin{pmatrix}
	\w_{1} & \w_{2}
	\end{pmatrix}$ is a tangent moving basis of $\bx$, $\phi = \lvert \det_{\tilde{\theta}}\tilde{\mathbf{c}} \rvert^{1/4}$ and $\tilde{\theta}$, $\tilde{h}_{ij}$ are associated to an arbitrary transversal vector field $\tilde{\bxi}$ whose induced affine fundamental form on $U \setminus \Sigma(\bx)$ is $\tilde{\mathbf{c}}$.
\end{teo}

\begin{proof}
Let $\bxi = \phi \tilde{\bxi} + a\w_{1} + b\w_{2}$ be the Blaschke vector field of $\bx$, where $\tilde{\bxi}$ is an arbitrary transversal vector field. It follows from definition \ref{defBlaschkeusual} that $\bxi$ is equiaffine and on its regular part the volume element $\theta$ induced by $\bxi$ coincides with the volume element given by the non-degenerate metric $\mathbf{c}$ (see (\ref{affinefundam})). The volume $\omega_{\mathbf{c}}$, obtained from $\mathbf{c}$ is given by 
\begin{align*}
\omega_{\mathbf{c}}(X_{1}, X_{2}) = \left| \det \mathbf{c}_{ij} \right|^{1/2},\; \text{where $\mathbf{c}_{ij} = \mathbf{c}(X_{i}, X_{j})$}.
\end{align*}
Thus $\theta = \omega_{\mathbf{c}}$ if and only if $\left|\det_{\theta}\mathbf{c}\right| = 1$, where $\det_{\theta}\mathbf{c}$ denotes $\det \mathbf{c}_{ij}$ when we take a unimodular basis relative to $\theta$, that is a basis $\lbrace X_{1}, X_{2} \rbrace$ such that $\theta(X_{1}, X_{2}) = 1$. If we denote by $\tilde{\mathbf{c}}$ the affine fundamental form induced by $\tilde{\bxi}$, that is
\begin{align*}
	D_{X}Y = \nabla_{X}Y + \tilde{\mathbf{c}}(X,Y)\tilde{\bxi},\; \text{on $U \setminus \Sigma(\bx)$},
\end{align*}
and by $\tilde{\theta}$ the volume element induced by $\tilde{\bxi}$, we get from $\bxi = \phi \tilde{\bxi} + a\w_{1} + b\w_{2}$ that $\mathbf{c} = \dfrac{1}{\phi}\tilde{\mathbf{c}}$ and $\theta = \phi \tilde{\theta}$ on $U \setminus \Sigma(\bx)$ . From this, one can show that
$\det_{\theta}\mathbf{c} = \phi^{-4}\det_{\tilde{\theta}}\tilde{\mathbf{c}}$.
Therefore, in order to have $\left|\det_{\theta}\mathbf{c}\right| = 1$, it is enough to take $\phi = \left| \det_{\tilde{\theta}}\tilde{\mathbf{c}} \right|^{1/4}$. 

%Since $\overline{\mathbf{c}}$ is the second fundamental form and a unimodular basis relative to $\overline{\theta}$ is being taken in order to obtain $\det_{\overline{\theta}}\overline{\mathbf{c}}$ (for instance, we can consider an orthonormal basis for the tangent space of $\bx$), we get $\phi  = \left| K \right|^{1/4}$.

Then, it follows from the density of $U \setminus \Sigma(\bx)$ and from the smoothness of $\phi$ that $\phi(q) = \lim\limits_{u \to q}\left| \det_{\tilde{\theta}}\tilde{\mathbf{c}} \right|^{1/4}$, where $q \in \Sigma(\bx)$. From the fact that $\bxi$ is transversal to $\bx$, we have that $\phi \neq 0$ on $U$, consequently $\det_{\tilde{\theta}}\tilde{\mathbf{c}}$ admits a non-vanishing extension to $U$. 

Now, let us take a look at the tangent component of $\bxi$. Since $\bxi$ is equiaffine, we have $\tau_{\Omega} \equiv 0$, thus taking into account the expression for $\tau_{\Omega}$ given in proposition \ref{prop1.3} we get
\begin{align*}
\tilde{h}_{\Omega}(a\w_{1} + b\w_{2},\w_{1}) + \phi_{u_{1}} &= a\tilde{h}_{11} + b\tilde{h}_{21} + \phi_{u_{1}} = 0,\\
\tilde{h}_{\Omega}(a\w_{1} + b\w_{2},\w_{2}) + \phi_{u_{2}} &= a\tilde{h}_{12} + b\tilde{h}_{22} + \phi_{u_{2}} = 0.
\end{align*}
As $\bx$ is non-degenerate on its regular part $U \setminus \Sigma(\bx)$, it follows from the above equations that for $p \in U \setminus \Sigma(\bx)$
\begin{align}\label{exptgblaschke}
\begin{pmatrix}
a\\ b
\end{pmatrix} &= \begin{pmatrix}
\tilde{h}_{11} & \tilde{h}_{21}\\
\tilde{h}_{12} & \tilde{h}_{22}
\end{pmatrix}^{-1} \begin{pmatrix}
-\phi_{u_1}\\
-\phi_{u_2}
\end{pmatrix}.
\end{align} 
Since $a, b \in C^{\infty}(U,\R)$ and $U \setminus \Sigma(\bx)$ is dense, we obtain for a point $q \in \Sigma(\bx)$ that
\begin{align*}
\begin{pmatrix}
a(q)\\ b(q)
\end{pmatrix} &= \lim\limits_{u \to q}\begin{pmatrix}
\tilde{h}_{11} & \tilde{h}_{21}\\
\tilde{h}_{12} & \tilde{h}_{22}
\end{pmatrix}^{-1} \begin{pmatrix}
-\phi_{u_1}\\
-\phi_{u_2}
\end{pmatrix}.
\end{align*} 
Reciprocally, considering $\delta$ the non-vanishing extension of $\det_{\tilde{\theta}}\tilde{\mathbf{c}}$, it is enough to define $\phi = \lvert \delta \rvert^{1/4}$ and take $a,b \in C^{\infty}(U,\R)$ satisfying (\ref{exptgblaschke}), then $\bxi$ is the Blaschke vector field of $\bx$.
\end{proof}

\begin{obs}\normalfont
	Note that theorem \ref{teoremacaracterizacaoBlaschke} could also be written in terms of the Gaussian curvature instead of $\det_{\overline{\theta}}\overline{\mathbf{c}}$, since the transversal vector field $\tilde{\bxi}$ was taken arbitrarily. 
\end{obs}

\subsection{Examples}\label{subex}
Now, we provide some examples taking into account the classes described in section \ref{sec3}. First, a remark.

\begin{obs}\normalfont\label{classesex}
	\hfill 	\begin{enumerate}
		\item [(a)]Let $\bx: U \to \R^3$ be a frontal in the class \ref{firstclass}, such that its Gaussian curvature $K$ admits a non-vanishing extension, so in order to have a Blaschke vector field we just need to verify the condition (\ref{exptgblaschkelim0}) in theorem \ref{teoremacaracterizacaoBlaschke}. One can verify that $K_{u_{2}} = d b_{u_{2}}$, for a smooth function $d$ is a sufficient condition for this to happen. For instance, with any of the choices below
		\begin{itemize}
			\item $b = u_{2}^2$, $r=0$, $l=1$ and $h = h(u_{1},u_{2})$ any smooth function (see example \ref{ex2}),
			\item $b = \dfrac{2}{5}u_{2}^5 + u_{2}^2$, $r=0$, $l=1$ and $h = h(u_{1},u_{2})$ any smooth function (see example \ref{ex1}),
		\end{itemize}
		we obtain $\bx$ for which $K$ admits a non-vanishing extension and condition (\ref{exptgblaschkelim0}) is verified, hence we have a Blaschke vector field.
		
		\item [(b)]	If $\bx: U \to \R^3$ is a wave front given by the class \ref{secondclass}, such that the Blaschke vector field exists in the regular part $U \setminus \Sigma(\bx)$, then it is given by $\bxi = (\bxi_1, \bxi_2, \bxi_3)$, where
		\begin{subequations}\label{blaschkeclasse1}
			\begin{align}
				\bxi_1 &= -\dfrac{1}{4}\dfrac{c_{u_1}}{c^{3/4}h_{u_{1}u_{1}}},\\
				\bxi_2 &= -\dfrac{1}{4}\dfrac{c_{u_2}h_{u_{1}u_{1}} - c_{u_1}h_{u_{1}u_{2}}}{c^{3/4}h_{u_{1}u_{1}}},\\
				\bxi_3 &= \dfrac{1}{4}\dfrac{-u_{2}c_{u_2}h_{u_{1}u_{1}} + u_{2}c_{u_1}h_{u_{1}u_{2}} + 4ch_{u_1u_1} - c_{u_1}h_{u_1}}{c^{3/4}h_{u_{1}u_{1}}}.
			\end{align} 
		\end{subequations}
		Note in (\ref{blaschkeclasse1}) that is not always possible to extend $\bxi$, since $h_{u_1u_{1}}(q) = 0$ for all $q \in \Sigma(\bx)$. However, if we take for instance, $c$ a smooth function such that $c_{u_{1}} = d h_{u_{1}u_{1}}$, for a smooth function $d$, then $\bxi$ admits an extension to the entire $U$. If $d=0$, we get $c = c(u_{2})$, which is satisfied for the case $c$ constant, for instance (see example \ref{ex3}). It is worth observing that for $c$ constant, we get $\bxi = (0,0,\rho)$, for some $\rho \in \R$ and if we think of frontal improper affine spheres as those frontals for which the Blaschke vector field is constant, then this choice of $c$ provides a class of these frontals. This is an important class, specially if we seek to understand frontals from the affine viewpoint, and will be further discussed in future works. It is also important to remark that improper affine spheres with singularities is a topic of interest in differential geometry, see \cite{esferas1}, \cite{ishikawaimproper} \cite{Martinez}, \cite{milan2013} and \cite{Nakajo}, for instance.
	\end{enumerate}
\end{obs}

\begin{ex}\label{ex2}\normalfont
	Let $\bx: U \to \R^3$ defined by $\bx = (u_{1}, u_{2}^2, {4/15}u_{1}u_{2}^5 + 1/2u_{1}^3u_{2}^4 + u_{1}u_{2}^2)$, where $U = (-1,1) \times (-4,4)$ (see figure \ref{frontalnp2}). This frontal is a cuspidal cross-cap obtained from a $5/2$-cuspidal edge and satisfies $\bx \isEquivTo{} \bn$, where $\bn$ is its unit normal. We decompose $D\bx = \bomega \bLambda^{T}$, where
	\begin{align*}
	\bomega = \begin{pmatrix}
	1 & 0\\
0 & 1\\
	u_{2}^2(4/15u_{2}^3 + 3/2u_{1}^2u_{2}^2+1) & 1/3u_{1}\left( 3u_{1}^2u_{2}^2 + 2u_{2}^3 + 3 \right) 	
	\end{pmatrix}\; \text{and}\; \bLambda = \begin{pmatrix}
	1 & 0\\
	0 & 2u_{2}
	\end{pmatrix}.
	\end{align*}
	We have that $\lambda_{\Omega} = 2u_{2}$ and $$K_{\Omega} = {\frac {18 \times 10^4u_{2}\left( 54{u_{1}}^{4}{u_{2}}^{4}+9{u_{1}}^{
				2}{u_{2}}^{5}+4{u_{2}}^{6}+54{u_{1}}^{2}{u_{2}}^{2}+12{u_{2}}^{3
			}+9 \right) }{ \mu^{2}}},$$ where
		\begin{align*}
			 \mu &= 2025{u_{1}}^{4}{u_{2}}^{8}+720{u_{1}}^{2}{u_
				{2}}^{9}+900{u_{1}}^{6}{u_{2}}^{4}+64{u_{2}}^{10}+1200{u_{1}}^{4
			}{u_{2}}^{5}+3100{u_{1}}^{2}{u_{2}}^{6}\\
		&+480{u_{2}}^{7}+1800{u_{1
			}}^{4}{u_{2}}^{2}+1200{u_{1}}^{2}{u_{2}}^{3}+900{u_{2}}^{4}+900{
				u_{1}}^{2}+900. 
			\end{align*}
				Therefore, considering that at a regular point the Gaussian curvature is given by $\dfrac{K_{\Omega}}{\lambda_{\Omega}}$, we obtain that the extension of the Gaussian curvature is \begin{align*}K= {\frac {9 \times 10^4\left( 54{u_{1}}^{4}{u_{2}}^{4}+9{u_{1}}^{
			 			2}{u_{2}}^{5}+4{u_{2}}^{6}+54{u_{1}}^{2}{u_{2}}^{2}+12{u_{2}}^{3
			 		}+9 \right) }{ \mu^{2}}}.\end{align*} Then, writing $\phi = \left|K \right|^{1/4}$, the Blaschke vector field of $\bx$ is given by $\bxi = \phi \bn + a\w_{1} + b\w_{2}$, where $a$ and $b$ are obtained using (\ref{exptgblaschke}). Thus, $\bxi = \dfrac{1}{ \rho^{7/4}}\left(\dfrac{-3\sqrt{3}}{8}\xi_1, \dfrac{9\sqrt{3}}{8}\xi_2, \dfrac{\sqrt{3}}{240}\xi_3 \right)$, where		 			
		\begin{align*}
		\xi_1 &= 216\,{u_{1}}^{6}{u_{2}}^{4}-189\,{u_{1}}^{4}{u_{2}}^{5}+
		66\,{u_{1}}^{2}{u_{2}}^{6}+16\,{u_{2}}^{7}+324\,{u_{1}}^{4}{u_{2}}^{2}\\
		&+9\,{u_{1}}^{2}{u_{2}}^{3}+48\,{u_{2}}^{4}+108\,{u_{1}}^{2}+36\,u_{2}
		\end{align*}
		\begin{align*}
			\xi_2 &= \left( 216\,{u_{1}}^{4}{u_{2}}^{4}+87\,{u_{1}}^{2}{u_{2}}^{5}-16\,{u_
				{2}}^{6}+252\,{u_{1}}^{2}{u_{2}}^{2}+24\,{u_{2}}^{3}+72 \right) {u_{2}
			}^{2}\\
			\xi_3 &= 145800\,{u_{1}}^{8}{u_{2}}^{8}+35721\,{u_{1}}^{6}{u_{2}}^{9}+25326\,{u
				_{1}}^{4}{u_{2}}^{10}+4896\,{u_{1}}^{2}{u_{2}}^{11}+277020\,{u_{1}}^{6
			}{u_{2}}^{6}\\ &+896\,{u_{2}}^{12}+114129\,{u_{1}}^{4}{u_{2}}^{7}+39204\,{
				u_{1}}^{2}{u_{2}}^{8}+5088\,{u_{2}}^{9}+179820\,{u_{1}}^{4}{u_{2}}^{4}
			+88938\,{u_{1}}^{2}{u_{2}}^{5}\\ &+12096\,{u_{2}}^{6}+48600\,{u_{1}}^{2}{u
				_{2}}^{2}+14040\,{u_{2}}^{3}+6480\\
			\rho &= 54\,{u_{1}}^{4}{u_{2}}^{4}+9\,{u_{1}}^{2}{u_{2}}^{5}+4\,{u_{2}}^{6}+54
			\,{u_{1}}^{2}{u_{2}}^{2}+12\,{u_{2}}^{3}+9.		 
		\end{align*}
		\begin{figure}[ht!]
			\includegraphics[scale=0.50]{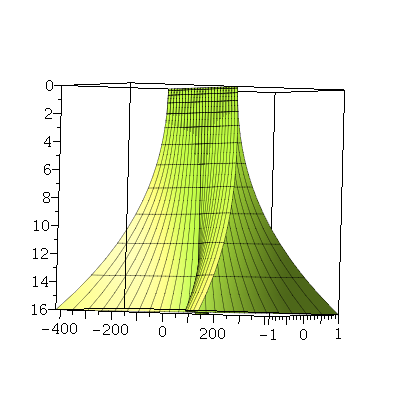}
			\includegraphics[scale=0.50]{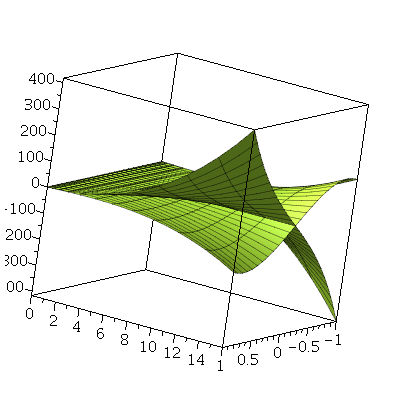}
			\caption{Frontal with extendable non-vanishing Gaussian curvature.}
			\label{frontalnp2}
		\end{figure}
\end{ex}

\begin{obs}\normalfont
It is worth observing that the Blaschke vector field $\bxi$ of a frontal is also a frontal, since $\bxi$ is equiaffine. In example \ref{ex2}, $\Sigma(\bxi)$ is given in figure \ref{singset}.
	\begin{figure}[ht!]
	\includegraphics[scale=0.30]{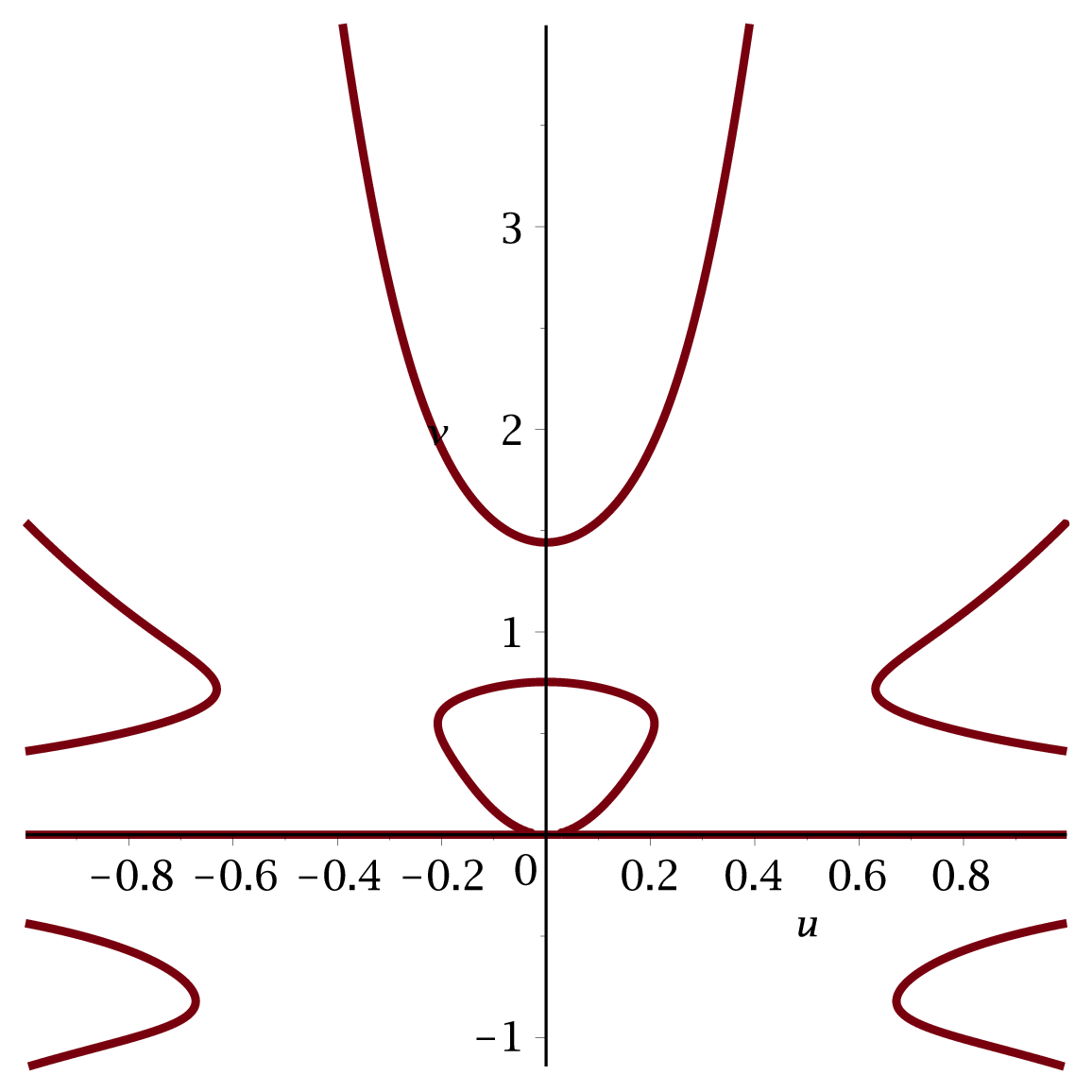}
	\caption{Singular set $\Sigma(\bxi)$ of the Blaschke vector field.}
	\label{singset}
\end{figure}
Then, the Blaschke vector field is a proper frontal and it is given in figure \ref{blaschke1}, restricting the domain to $(-1/10, 1/10) \times (-1/60, 1/60)$.
	\begin{figure}[ht!]
	\includegraphics[scale=0.55]{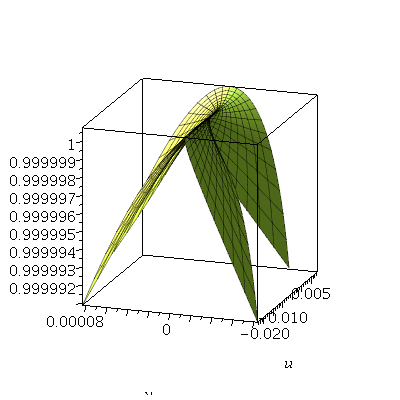}
	\includegraphics[scale=0.50]{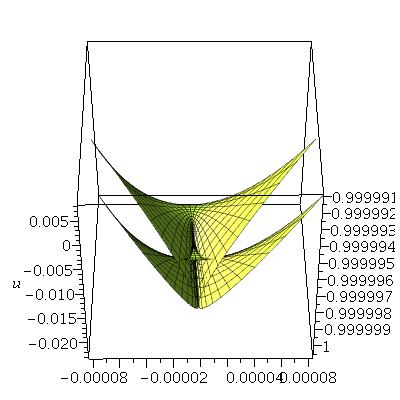}
	\caption{The Blaschke vector field from example \ref{ex2}}
	\label{blaschke1}
\end{figure}
\end{obs}

\begin{ex} \label{ex1}\normalfont
	Let $\bx: U \rightarrow \R^3$ defined by $\bx = \left(u_{1},  \frac{2}{5}u_{2}^5 + u_{2}^2, u_{1}u_{2}^2  \right)$, for $ U = \left(-1,1 \right) \times (-1, 1)$ (see figure \ref{frontalnp}). This frontal satisfies $\bx \isEquivTo{} \bn$, where $\bn$ is its unit normal. We decompose $D\bx = \bomega \bLambda^{T}$, where
	\begin{align*}
	\bomega = \begin{pmatrix}
	1 & 0\\
	0 & u_{2}^3 + 1\\
	u_{2}^2 & u_{1}
	\end{pmatrix}\; \text{and}\; \bLambda = \begin{pmatrix}
	1 & 0\\
	0 & 2u_{2}
	\end{pmatrix}.
	\end{align*}
	We have that $\lambda_{\Omega} = 2u_{2}$ and 
	\begin{align*}
	K_{\Omega} = \dfrac{2u_{2}(u_{2} + 1)^2(u_{2}^2 - u_{2} + 1)^2}{(u_{2}^{10} + 2u_{2}^7 + u_{2}^6 + u_{2}^4 + 2u_{2}^3 + u_{1}^2 + 1)^2}.
	\end{align*}
	Then,
	\begin{align*}
	K = \dfrac{(u_{2} + 1)^2(u_{2}^2 - u_{2} + 1)^2}{(u_{2}^{10} + 2u_{2}^7 + u_{2}^6 + u_{2}^4 + 2u_{2}^3 + u_{1}^2 + 1)^2}
	\end{align*}
	is the extension of the Gaussian curvature to $U$. In a similar way to example \ref{ex2}, we obtain that the Blaschke vector field of $\bx$ is given by
	\begin{align*}
	\bxi = \frac{1}{4(u_{2}^2 + u_{2} + 1)^{3/2}(u_{2}+1)^{3/2}}(3u_{2}, 0, 7u_{2}^3 + 4).
	\end{align*}
	\begin{figure}[ht!]
		\includegraphics[scale=0.50]{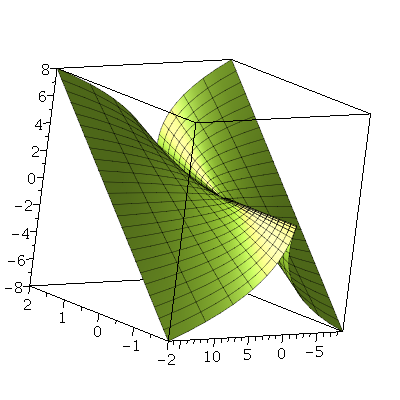}
		\includegraphics[scale=0.50]{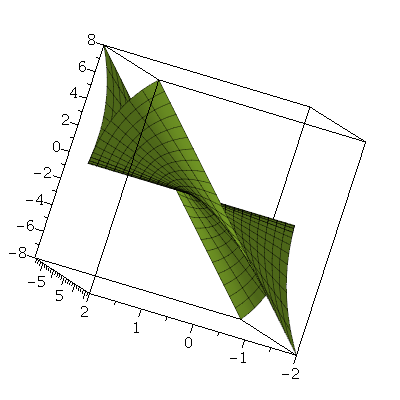}
		\caption{Frontal with extendable non-vanishing Gaussian curvature.}
		\label{frontalnp}
	\end{figure}
	
\end{ex}

\begin{ex}\normalfont\label{ex3}
	Let $\bx: \R^2 \to \R^3$ defined by $\bx = (u_{1}, 12\,{u_{1}}^{2}u_{2}-4\,{u_{2}}^{3}, {u_{1}}^{4}+6\,{u_{1}}^{2}{u_{2}}^{2}-3\,{u_{2}}^{4})$ (see figure \ref{frontalnp3}). This is a wave front of rank $1$ for which the Gaussian curvature admits a non-vanishing extension. We decompose $D\bx = \bomega \bLambda^{T}$, where
	\begin{align*}
	\bomega = \begin{pmatrix}
	1 & 0\\
	24u_{1}u_{2} & 1\\
	4u_{1}^3 + 12u_{1}u_{2}^2 & u_{2} 	
	\end{pmatrix}\; \text{and}\; \bLambda = \begin{pmatrix}
	1 & 0\\
	0 & 12u_{1}^2 - 12u_{2}^2
	\end{pmatrix}.
	\end{align*}
	We have that $\lambda_{\Omega} = -12u_{1}^2 + 12u_{2}^2$ and 
	\begin{align*}
	K_{\Omega} = -{\dfrac {12({u_{1}}^{2}-{u_{2}}^{2})}{ \left( 16\,{u_{1}}^{6}-96\,{u_{1
			}}^{4}{u_{2}}^{2}+144\,{u_{1}}^{2}{u_{2}}^{4}+{u_{2}}^{2}+1 \right) ^{
				2}}}.
	\end{align*}
	Then, the extension of the Gaussian curvature is
	\begin{align*}
	K = {\dfrac {-1}{ \left( 16\,{u_{1}}^{6}-96\,{u_{1
			}}^{4}{u_{2}}^{2}+144\,{u_{1}}^{2}{u_{2}}^{4}+{u_{2}}^{2}+1 \right) ^{
				2}}}.
	\end{align*}
	In a similar way to example \ref{ex2}, we obtain that the Blaschke vector field of $\bx$ is given by $\bxi = (0,0,1)$.
	\begin{figure}[ht!]
		\includegraphics[scale=0.50]{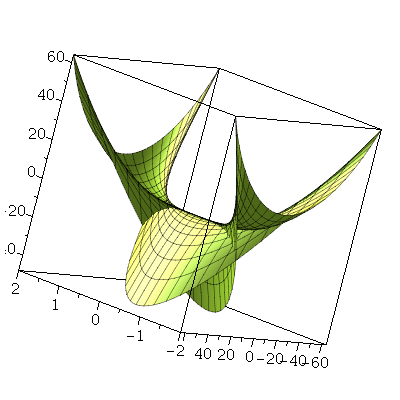}
		\includegraphics[scale=0.50]{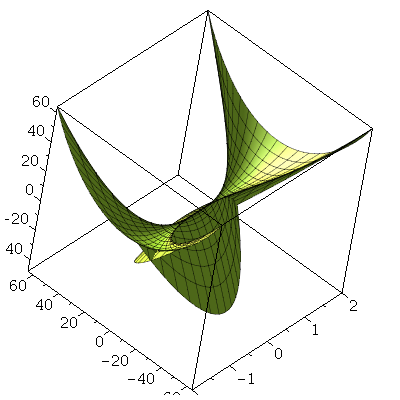}
		\caption{Wave front of rank $1$ with extendable non-vanishing Gaussian curvature.}
		\label{frontalnp3}
	\end{figure}
\end{ex}

\subsection{Conormal vector field}
\begin{defi}\normalfont
	Given a frontal $\bx: U \rightarrow \R^3$, we define the affine conormal vector field of $\bx$ relative to an equiaffine transversal vector field $\bxi$ as the vector field $\bnu: U \rightarrow \R^3 \setminus \bm{0}$, such that
	\begin{align*}
	\langle \bnu(u), \bxi(u) \rangle &= 1\\
	\langle \bnu(u), \w \rangle &= 0,\; \text{for all $\w \in T_{\Omega}(u)$},
	\end{align*}
	for all $u \in U$.
\end{defi}

\begin{obs}\label{obsortonormal}\normalfont
	It follows from the second condition above that the conormal vector field is always a multiple of the unit normal vector field $\bn$ of $\bx$. 
\end{obs}	

The next proposition shows that an important property of the conormal vector field is still valid when we are working with the case of non-parabolic frontals.
\begin{prop}
	Given a frontal $\bx: U \rightarrow \R^3$, a \text{tmb} $\bomega$ and an equiaffine transversal vector field $\bxi$, we have that
	\begin{align*}
	\langle \bnu_{{u_{i}}}, \bxi \rangle &= 0,\, i=1,2.\\
	\langle \bnu_{{u_{i}}}, \vv  \rangle &= -h_{\Omega}(\vv, \w_{i}),\, \text{where $\vv(u) \in T_{\Omega}(u)$ for all $u \in U$},\, i,j=1,2.
	\end{align*}
	Furthermore, if $\bx$ is non-parabolic then the conormal vector field is an immersion.
\end{prop}

\begin{proof}
	In order to simplify
notation, we drop the subscript in the notation for the induced affine fundamental form and indicate only by $h$. We know that $\langle \bnu, \bxi \rangle = 1$ so, differentiating and using the fact that $\bxi$ is equiaffine, we get $\langle \bnu_{u_{i}}, \bxi \rangle = 0$. 
	
	From $\langle \bnu, \vv \rangle = 0$, it follows that $\langle \bnu_{u_i}, \vv \rangle = - \langle \bnu, \vv_{u_i} \rangle  $. We can write
	\begin{align*}
	-\vv_{u_i} = -\vv_{u_i}^{\top} - h(\vv,\w_{i})\bxi,
	\end{align*}
	where $\vv_{u_i}^{\top}$ indicates the tangent component of $\vv_{u_i}$, thus considering the properties of the conormal vector field, we obtain
	\begin{align}\label{conormalex}
\langle \bnu_{u_i}, \vv \rangle = 	-	\langle \bnu, \vv_{u_i} \rangle = -h(\vv,\w_{i}).
	\end{align}
	
	Now, let us suppose that $\bx$ is non-parabolic and that $\bnu$ is not an immersion, hence there is $(a,b) \in \R^2 \setminus 0$ such that $a \bnu_{u_1} + b\bnu_{u_2} = 0$. Thus
	\begin{align*}
	\langle a \bnu_{u_1} + b\bnu_{u_2}, \vv  \rangle &= 0.
	\end{align*}
	By using (\ref{conormalex}) and the above expression, we get
	\begin{align*}
	0 = -ah(\vv, \w_{1}) - bh(\vv,\w_{2}) = h(\vv, -a\w_{1} - b\w_{2}),\, \text{for all $\vv$},
	\end{align*}
	but this contradicts the fact that $h$ is non-degenerate.
\end{proof}

%\begin{cor}
%	Two non-parabolic frontals with a common equiaffine transversal vector field, with the same associated conormal vector field and the same induced affine fundamental form are affine equivalent.
%\end{cor}
%\begin{proof}
%	Let $\bx_{1}, \bx_{2}: U \rightarrow \R^3$ be two non-parabolic frontals and $\bxi: U \rightarrow \R^3$ a common equiaffine transversal vector field. Let us suppose that both non-parabolic frontals have the same conormal vector field $\bnu$ associated to $\bxi$ (so, they have the same unit normal vector field) and $h_{\Omega}$ as affine fundamental form induced by $\bxi$ relative to a \text{tmb} $\bomega$. Then, from the last proposition we have that
%	\begin{align*}
%	\langle \bnu, \bx_{1u_i} \rangle &= \langle \bnu, \bx_{2u_i} \rangle \\
%	\langle \bnu_{{u_{j}}},  \bx_{1u_i} \rangle & = \langle \bnu_{{u_{j}}},  \bx_{2u_i} \rangle,
%	\end{align*}
%	for $i,j=1,2$. It follows from above that $\bx_{1u_i} -  \bx_{2u_i}$ is perpendicular to $\bnu, \bnu_{{u_{1}}}$ and $\bnu_{{u_{2}}}$. But, $\bnu$ is an immersion, then $ \bx_{1u_i} -  \bx_{2u_i} = 0$, $i,j=1,2$. Therefore, 
%	\begin{align*}
%	D\bx_1 = D\bx_2,
%	\end{align*}
%which means that $\bx_{1}$ and $\bx_{2}$ are affine equivalent.
%\end{proof}

\section{A fundamental theorem}\label{sec6}
In this section we provide, in theorem \ref{teofundamental}, a fundamental theorem for the theory developed in section \ref{sec2}. This theorem is a version for frontals of the fundamental theorem of affine differential geometry (see section 4.9 in \cite{Simon} for the classical result for regular surfaces). Thus, taking $U \subset \R^2$ an open subset and assuming the integrability conditions for the regular case are valid in an open dense subset of $U$, we obtain for each $q \in U$ a neighborhood $V \subset U$ of $q$, a frontal $\bx: V \to \R^3$ and an equiaffine transversal vector field $\bxi: V \to \R^3$, in the sense of section \ref{sec2}. In order to do this, we use the same approach applied in \cite{alexandro2019fundamental}.
\subsection{The compatibility equations}
Let $\bx: U\rightarrow \R^3$ be a proper frontal, $\bomega = \begin{pmatrix}
\w_{1} & \w_{2}
\end{pmatrix}$ a tmb and $\bn = \frac{\w_{1} \times \w_{2}}{\lVert \w_{1} \times \w_{2} \rVert}$ the unit normal vector field induced by $\bomega$. By considering the decomposition $D\bx = \bomega \bLambda^{T}$ and $\lambda_{\Omega} = \det \bLambda$, one can get on $ U \setminus \lambda_{\Omega}^{-1}(0)$ the following structural equations
\begin{subequations}\label{eqcompa00}
	\begin{align}\label{eqcompa001}
	\bx_{u_{1}u_1} &= \Gamma^{1}_{11}\bx_{u_1} + \Gamma^{2}_{11}\bx_{u_2} + e\bn\\
	\bx_{u_1u_{2}} &= \Gamma^{1}_{21}\bx_{u_1} + \Gamma^{2}_{21}\bx_{u_2} +f\bn\\
	\bx_{u_2u_2} &= \Gamma^{1}_{22}\bx_{u_1} + \Gamma^{2}_{22}\bx_{u_2} + g\bn,
	\end{align}
\end{subequations}
where $e$, $f$ and $g$ are the coefficients of the second fundamental form of $\bx$. In a similar way, if $\bxi: U \rightarrow \R^3$ is an equiaffine transversal vector field to $\bx$, the following hold on $U \setminus \lambda_{\Omega}^{-1}(0)$  		
\begin{subequations}\label{eqcompa0}
	\begin{align}\label{eqcompa01}
	\bx_{u_{1}u_1} &= \widetilde{\Gamma}^{1}_{11}\bx_{u_1} + \widetilde{\Gamma}^{2}_{11}\bx_{u_2} + c_{11}\bxi\\
	\bx_{u_1u_{2}} &= \widetilde{\Gamma}^{1}_{21}\bx_{u_1} + \widetilde{\Gamma}^{2}_{21}\bx_{u_2} + c_{21}\bxi\\
	\bx_{u_2u_2} &= \widetilde{\Gamma}^{1}_{22}\bx_{u_1} + \widetilde{\Gamma}^{2}_{22}\bx_{u_2} + c_{22}\bxi\\
	\bxi_{u_{1}} &= -b^{1}_{1}\bx_{u_1} - b^{2}_{1}\bx_{u_2}\\
	\bxi_{u_{2}} &= -b^{1}_{2}\bx_{u_1} - b^{2}_{2}\bx_{u_2}.
	\end{align}
\end{subequations}	
The symbols $c_{ij}$ induce a symmetric bilinear form called the affine fundamental $\mathbf{c}$ relative to $\bxi$, while the symbols $ \widetilde{\Gamma}^{k}_{ij}$ are associated to the induced affine connection $\nabla$ (see (\ref{affinefundam})).

\begin{prop}\label{propcoef1}
If we write $\bxi = \phi \bn + a \bx_{u_1} + b \bx_{u_2}$, where $\phi \neq 0$, then on $U \setminus \Sigma(\bx)$ we have:
	\begin{align}
	\begin{pmatrix} 
	c_{11} & c_{12}\\
	c_{12} & c_{22}
	\end{pmatrix} &= \dfrac{1}{\phi} 	\begin{pmatrix}
	e &f\\
	f & g
	\end{pmatrix}\label{propext},\\
	\widetilde{\Gamma}_{1} &= \begin{pmatrix}
	\widetilde{\Gamma}^{1}_{11} & \widetilde{\Gamma}^{2}_{11}\\
	\widetilde{\Gamma}^{1}_{21} & \widetilde{\Gamma}^{2}_{21}	
	\end{pmatrix} =  \begin{pmatrix}
	\Gamma^{1}_{11} & \Gamma^{2}_{11}\\
	\Gamma^{1}_{21} & \Gamma^{2}_{21}	
	\end{pmatrix} - \dfrac{1}{\phi}\begin{pmatrix}
	ae & be\\
	af & bf	\end{pmatrix} = \Gamma_{1} - 	\dfrac{1}{\phi} \begin{pmatrix}
	ae & be\\
	af & bf	\end{pmatrix},\label{propext1}\\
\widetilde{\Gamma}_{2} &= \begin{pmatrix}
	\widetilde{\Gamma}^{1}_{21} & \widetilde{\Gamma}^{2}_{21}\\
	\widetilde{\Gamma}^{1}_{22} & \widetilde{\Gamma}^{2}_{22}	
	\end{pmatrix} =  \begin{pmatrix}
	\Gamma^{1}_{21} & \Gamma^{2}_{21}\\
	\Gamma^{1}_{22} & \Gamma^{2}_{22}	
	\end{pmatrix} - \dfrac{1}{\phi} \begin{pmatrix}
	af & bf\\
	ag & bg	\end{pmatrix} = \Gamma_{2} - \dfrac{1}{\phi}	\begin{pmatrix}
	af & bf\\
	ag & bg	\end{pmatrix}\label{propext2}.
\end{align}
\end{prop}
\begin{proof}
If $\bxi = \phi \bn + a \bx_{u_{1}} + b \bx_{u_2}$ in (\ref{eqcompa0}), then
\begin{align*}
\bx_{u_{1}u_1} &= (\widetilde{\Gamma}^{1}_{11} + c_{11}a)\bx_{u_1} + (\widetilde{\Gamma}^{2}_{11} + c_{11}b)\bx_{u_2} + \phi c_{11}\bn\\
\bx_{u_1u_{2}} &= (\widetilde{\Gamma}^{1}_{21} + c_{21}a)\bx_{u_1} + (\widetilde{\Gamma}^{2}_{21} + c_{21}b)\bx_{u_2} + \phi c_{21}\bn\\
\bx_{u_2u_2} &= (\widetilde{\Gamma}^{1}_{22} + c_{22}a)\bx_{u_1} + (\widetilde{\Gamma}^{2}_{22} + c_{22}b)\bx_{u_2} + \phi c_{22}\bn.
\end{align*}	
Hence, comparing this to (\ref{eqcompa00}) we have the result.
\end{proof}

If we look at the basis $\w_{1}, \w_{2}, \bn$ of $\R^3$, it follows from (\ref{decompo2}) that there are smooth functions $\cT^{k}_{ij}$ defined on $U$, $i,j,k \in \lbrace 1,2 \rbrace $, such that
\begin{subequations}\label{eqcompanormal}
	\begin{align}\label{eqcompanormal1}
	\w_{1u_1} &= \cT^{1}_{11}\w_{1} + \cT^{2}_{11}\w_2 + e_{\Omega}\bn\\
	\w_{2u_1} &= \cT^{1}_{21}\w_{1} + \cT^{2}_{21}\w_2 + f_{2\Omega}\bn\\
	\w_{1u_2} &= \cT^{1}_{12}\w_{1} + \cT^{2}_{12}\w_2 + f_{1\Omega}\bn\\
	\w_{2u_2} &= \cT^{1}_{22}\w_{1} + \cT^{2}_{22}\w_2 + g_{\Omega}\bn
	\end{align}
\end{subequations}
We know that $\w_{1}, \w_{2}, \bxi$ is a basis of $\R^3$, then from (\ref{decompo4}), (\ref{decompo5}) and from the fact that $\bxi$ is equiaffine there are smooth functions $\cD^{k}_{ij}$, $h_{ij}$ and $S^{i}_{j}$ defined on $U$, $i,j,k \in \lbrace 1,2 \rbrace $, such that	
\begin{subequations}\label{eqcompa}
	\begin{align}\label{eqcompa1}
	\w_{1u_1} &= \cD^{1}_{11}\w_{1} + \cD^{2}_{11}\w_2 + h_{11}\bxi\\
	\w_{2u_1} &= \cD^{1}_{21}\w_{1} + \cD^{2}_{21}\w_2 + h_{21}\bxi\\
	\w_{1u_2} &= \cD^{1}_{12}\w_{1} + \cD^{2}_{12}\w_2 + h_{12}\bxi\\
	\w_{2u_2} &= \cD^{1}_{22}\w_{1} + \cD^{2}_{22}\w_2 + h_{22}\bxi\\
	\bxi_{u_{1}} &= -S^{1}_{1}\w_{1} - S^{2}_{1}\w_{2}\\
	\bxi_{u_{2}} &= -S^{1}_{2}\w_{1} - S^{2}_{2}\w_{2}.
	\end{align}
\end{subequations}
If we write $\bxi = \phi \bn + a \bx_{u_1} + b\bx_{u_2}$ and $\bxi = \phi \bn + \tilde{a} \w_{1} + \tilde{b} \w_{2}$, where $\phi \neq 0$, then on $U \setminus \Sigma(\bx)$
\begin{align}\label{Rel1}
\begin{pmatrix}
\tilde{a} \\
\tilde{b}
\end{pmatrix} = \bLambda^{T} \begin{pmatrix}
a \\
b
\end{pmatrix}\; \text{and}\;
\begin{pmatrix}
\tilde{a}e_{\Omega} & \tilde{b}e_{\Omega}\\
\tilde{a}f_{2\Omega} & \tilde{b}f_{2\Omega}	\end{pmatrix} = \begin{pmatrix}
ae_{\Omega} & be_{\Omega}\\
af_{2\Omega} & bf_{2\Omega}	\end{pmatrix}\bLambda.
\end{align}
Hence, taking (\ref{eqcompanormal}), (\ref{eqcompa}), (\ref{Rel1}) and using the method of the proof of proposition \ref{propcoef1} one can get the following result:
\begin{prop}\label{propnew2}
If we write $\bxi = \phi \bn + \tilde{a} \w_{1} + \tilde{b} \w_{2}$, for $\phi, \tilde{a},\tilde{b} \in C^{\infty}(U, \R)$ and $\phi \neq 0$, then
\begin{align}
\begin{pmatrix}
h_{11} & h_{12}\\
h_{21} & h_{22}
\end{pmatrix} &= \dfrac{1}{\phi}\begin{pmatrix}
e_{\Omega} & f_{1\Omega}\\
f_{2\Omega} & g_{\Omega}
\end{pmatrix}\; \text{on $U$},\label{relacaoh}\\
\cD_{1} &= \cT_{1} -  \dfrac{1}{\phi}\begin{pmatrix}
\tilde{a}e_{\Omega} & \tilde{b}e_{\Omega}\\
\tilde{a}f_{2\Omega} & \tilde{b}f_{2\Omega}	\end{pmatrix} = \cT_{1} - \dfrac{1}{\phi} \begin{pmatrix}
ae_{\Omega} & be_{\Omega}\\
af_{2\Omega} & bf_{2\Omega}	\end{pmatrix}\bLambda\; \text{on $U \setminus \Sigma(\bx)$}, \label{expd1}\\
\cD_{2} &= \cT_{2} -  \dfrac{1}{\phi}\begin{pmatrix}
\tilde{a}f_{1\Omega} & \tilde{b}f_{1\Omega}\\
\tilde{a}g_{\Omega} & \tilde{b}g_{\Omega}
\end{pmatrix} = \cT_{2} - \dfrac{1}{\phi} \begin{pmatrix}
af_{1\Omega} & bf_{1\Omega}	\\
ae_{\Omega} & be_{\Omega}
\end{pmatrix}\bLambda\; \text{on $U \setminus \Sigma(\bx)$}\nonumber,
\end{align}
where $\cD_{1} = \begin{pmatrix}
\cD^{1}_{11} & \cD^{2}_{11}\\
\cD^{1}_{21} & \cD^{2}_{21}
\end{pmatrix}$,  $\cD_{2} = \begin{pmatrix}
\cD^{1}_{12} & \cD^{2}_{12}\\
\cD^{1}_{22} & \cD^{2}_{22}
\end{pmatrix}$, $\cT_{1} = \begin{pmatrix}
\cT^{1}_{11} & \cT^{2}_{11}\\
\cT^{1}_{21} & \cT^{2}_{21}	
\end{pmatrix}$ and $\cT_{2} = \begin{pmatrix}
\cT^{1}_{12} & \cT^{2}_{12}\\
\cT^{1}_{22} & \cT^{2}_{22}	
\end{pmatrix}$.
\end{prop}

\begin{obs}\normalfont
With notation as in proposition (\ref{propnew2}), it follows from $\bII = \bLambda \bII_{\Omega}$ (see \ref{segform}), from \ref{propext} and from (\ref{relacaoh}) that
\begin{align}\label{decompmatrizaux}
\begin{pmatrix}
ae & be \\
af & bf
\end{pmatrix} &= \begin{pmatrix}
e & e \\
f & f
\end{pmatrix}\begin{pmatrix}
a & 0\\
0 & b
\end{pmatrix}
= \bLambda \begin{pmatrix}
e_{\Omega} & e_{\Omega} \\
f_{2\Omega} & f_{2\Omega}
\end{pmatrix}\begin{pmatrix}
a & 0\\
0 & b
\end{pmatrix},\\
\begin{pmatrix}
c_{11} & c_{12}\\
c_{12} & c_{22}
\end{pmatrix} &= \frac{1}{\phi} 	\begin{pmatrix}
e &f\\
f & g
\end{pmatrix} = \frac{1}{\phi} \bLambda\begin{pmatrix}
e_{\Omega} &f_{2\Omega}\\
f_{1\Omega} & g_{\Omega}
\end{pmatrix} = \bLambda \begin{pmatrix}
h_{11} & h_{12}\\
h_{12} & h_{22}
\end{pmatrix}\nonumber
\end{align}	
on $U \setminus \Sigma(\bx)$. Furthermore, just taking the decomposition $D\bx = \bomega \bLambda^{T}$, we have that
\begin{align*}
\begin{pmatrix}
S_{1}^{1} & S^{2}_{1}\\
S_{2}^{1} & S^{2}_{2}
\end{pmatrix} &=   \begin{pmatrix}
b_{1}^{1} & b_{1}^{2}\\
b_{2}^{1} & b^{2}_{2}
\end{pmatrix}\bLambda.
\end{align*}	

\end{obs}

\begin{prop}
Let $\bx: U \rightarrow \R^3$ be a proper frontal and $\bomega$ a tmb of $\bx$. Then, on $U \setminus \Sigma(\bx)$, we can write 
\begin{align*}
\cD_{1} &= \bLambda^{-1}\left(\widetilde{\Gamma}_{1}\bLambda - \bLambda_{u_1} \right),\\
\cD_{2} &=\bLambda^{-1}\left(\widetilde{\Gamma}_{2}\bLambda - \bLambda_{u_2} \right).
\end{align*}
	
\end{prop}

\begin{proof}
	It is known that the Christoffel symbols for the decomposition in the basis $\lbrace \bx_{u_1}, \bx_{u_2}, \bn \rbrace$ are given by $\Gamma_{1} = \left(D\bx_{u_1}^{T}D\bx  \right)\bI^{-1}$ (see section 4.3 in \cite{do2016differential}) thus, by taking (\ref{propext}), we get
		\begin{align*}
	\widetilde{\Gamma}_{1} = \Gamma_{1} - \dfrac{1}{\phi}\begin{pmatrix}
	ae & be \\
	af & bf
	\end{pmatrix} &= (D\bx_{u_1}^{T}D\bx)\mathbf{I}^{-1} - \dfrac{1}{\phi}\begin{pmatrix}
	ae & be \\
	af & bf
	\end{pmatrix},\; \text{from $D\bx = \bomega\bLambda^{T}$, we have} \\ &= \left((\bLambda_{u_1}\bomega^{T} + \bLambda \bomega_{u_1}^{T}) \bomega \bLambda^{T} \right)(\bLambda^{T})^{-1}\mathbf{I}_{\Omega}\bLambda^{-1} - \dfrac{1}{\phi}\bLambda\begin{pmatrix}
	ae_{\Omega} & be_{\Omega}\\
	af_{2\Omega} & bf_{2\Omega}	\end{pmatrix}\\
	&= \left(\bLambda_{u_1}\bomega^{T}\bomega + \bLambda \bomega_{u_1}^{T}\bomega \right)(\bomega^{T}\bomega)^{-1} \bLambda^{-1} - \dfrac{1}{\phi} \bLambda\begin{pmatrix}
	ae_{\Omega} & be_{\Omega}\\
	af_{2\Omega} & bf_{2\Omega}	\end{pmatrix} \bLambda \bLambda^{-1}\\
	&= \left( \bLambda_{u_1} + \bLambda\left(\cT_{1} - \dfrac{1}{\phi}\begin{pmatrix}
	ae_{\Omega} & be_{\Omega}\\
	af_{2\Omega} & bf_{2\Omega}	\end{pmatrix} \bLambda\right) \right) \bLambda^{-1},\; \text{from (\ref{expd1}), we have}\\
	&= (\bLambda_{u_1} + \bLambda\cD_{1})\bLambda^{-1},
	\end{align*}
since $\cT_{1} = \bomega_{u_{1}}^{T} \bomega (\bomega^{T}\bomega)^{-1}$. From this, we obtain  $\cD_{1} = \bLambda^{-1}\left(\widetilde{\Gamma}_{1}\bLambda - \bLambda_{u_1} \right)$. Similarly, we prove the other case.
\end{proof}	

From now on, if $\mathbf{A} \in M_{n\times n}(\R)$, we denote by $\mathbf{A}_{(i)}$ the ith-row and by $\mathbf{A}^{(j)}$ the jth-column of $\mathbf{A}$.

\begin{prop}\label{propextensao}
Let $\bI, \bI_{\Omega}, \bLambda, \begin{pmatrix}
h_{ij}
\end{pmatrix}, \begin{pmatrix}
c_{ij}
\end{pmatrix}: U \rightarrow M_{2\times 2}(\R)$ be arbitrary smooth maps, such that $\bI_{\Omega}$ is symmetric and $i,j =1,2 $. Consider also $\lambda_{\Omega} = \det \bLambda$ and $\mathfrak{T}_{\Omega}$ the principal ideal generated by $\lambda_{\Omega}$ in the ring $C^{\infty}(U, \R)$. Suppose that $U \setminus \lambda_{\Omega}^{-1}(0)$ is an open dense set and that
\begin{align}\label{eqdecI}
\mathbf{I} &= \begin{pmatrix}
E & F\\
F & G
\end{pmatrix} = \bLambda \mathbf{I}_{\Omega}\bLambda^{T}\\
(c_{ij}) &=  \frac{1}{\phi} 	\begin{pmatrix}
e &f\\
f & g
\end{pmatrix} = \bLambda  \begin{pmatrix}
h_{ij}
\end{pmatrix}\nonumber,
\end{align}
where $\phi \in C^{\infty}(U, \R \setminus 0)$ and define on $U \setminus \lambda_{\Omega}^{-1}(0)$, $\widetilde{\Gamma}_{1}$ and $\widetilde{\Gamma}_{2}$ by (\ref{propext1}) and (\ref{propext2}), respectively. Then,
\begin{enumerate}[(a)]
	\item The map $
	\bLambda^{-1}\left(\widetilde{\Gamma}_{1}\bLambda - \bLambda_{u_1} \right): U \setminus \lambda_{\Omega}^{-1}(0) \to M_{2 \times 2}(\R)\label{ext1}
$
has a unique $C^{\infty}$ extension to $U$ if and only if,
\begin{align}
	\bLambda_{(1)u_{1}}^{\phantomsection}\mathbf{I}_\Omega\bLambda_{(2)}^T-\bLambda_{(1)}^{\phantomsection}\mathbf{I}_\Omega\bLambda_{(2)u_{1}}^T+E_{u_{2}}-F_{u_{1}} &\in \mathfrak{T}_\Omega\label{cli1}.
\end{align}

\item The map $
\bLambda^{-1}\left(\widetilde{\Gamma}_{2}\bLambda - \bLambda_{u_2} \right): U \setminus \lambda_{\Omega}^{-1}(0) \to M_{2 \times 2}(\R)\label{ext2}
$
has a unique $C^{\infty}$ extension to $U$ if and only if,
\begin{align*}
\bLambda_{(1)u_{2}}^{\phantomsection}\mathbf{I}_\Omega\bLambda_{(2)}^T-\bLambda_{(1)}^{\phantomsection}\mathbf{I}_\Omega\bLambda_{(2)u_{2}}^T+F_{u_{2}}-G_{u_{1}} &\in \mathfrak{T}_\Omega.
\end{align*}\label{itembb}
\end{enumerate}
\end{prop}

\begin{proof}
	\begin{enumerate}[(a)]
	\item Let us suppose $\cD_{1}$ the $C^{\infty}$ extension of $\bLambda^{-1}\left(\widetilde{\Gamma}_{1}\bLambda - \bLambda_{u_1} \right)$ to $U$, thus 
	\begin{align*}
	\bLambda \cD_{1} = \widetilde{\Gamma}_{1}\bLambda - \bLambda_{u_1}
	\end{align*}	
	on $U \setminus \lambda_{\Omega}^{-1}(0)$. From (\ref{propext1}), it follows that
	\begin{align*}
	\bLambda \cD_{1} = \Gamma_{1}\bLambda - \dfrac{1}{\phi}\begin{pmatrix}
	ae & be\\
	af & bf	\end{pmatrix} \bLambda - \bLambda_{u_1}.
	\end{align*} 
	It is known that $\Gamma_{1} = (\frac{1}{2}\mathbf{I}_{u_1}+\frac{1}{2}\mathbf{A}_1)\mathbf{I}^{-1}$, where $\mathbf{A}_{1} = \begin{pmatrix}
	0 & - \left(E_{v} - F_{u}  \right)\\
	E_{v} - F_{u} & 0
	\end{pmatrix}$. Hence
	\begin{align*}
	\bLambda \cD_{1} = (\frac{1}{2}\mathbf{I}_{u_1}+\frac{1}{2}\mathbf{A}_1)\mathbf{I}^{-1}\bLambda - \dfrac{1}{\phi}\begin{pmatrix}
	ae & be\\
	af & bf	\end{pmatrix} \bLambda - \bLambda_{u_1}.
	\end{align*} 
	Via (\ref{eqdecI}) we have an expression for $\mathbf{I}_{u_{1}}$, then multiplying the above expression by the right side with $2\mathbf{I}_{\Omega}\bLambda^{T}$ and taking into account (\ref{expd1}) and (\ref{decompmatrizaux}), we obtain that
	\begin{align}
	\bLambda\left( 2\cD_{1}\mathbf{I}_{\Omega} - \mathbf{I}_{\Omega u_{1}} + \dfrac{2}{\phi}\begin{pmatrix}
	\tilde{a}e_{\Omega} & \tilde{b}e_{\Omega} \\
	\tilde{a}f_{2\Omega} & \tilde{b}f_{2\Omega}
	\end{pmatrix} \mathbf{I}_{\Omega} \right) \bLambda^{T}= \bLambda\mathbf{I}_\Omega\bLambda_{u_1}^t-\bLambda_{u_1}\mathbf{I}_\Omega\bLambda^t+\mathbf{A}_1.\label{equi2}
	\end{align}
	Note that the right side of the above equation is a skew-symmetric matrix, which implies that $ 2\cD_{1}\mathbf{I}_{\Omega} - \mathbf{I}_{\Omega u_{1}} + \dfrac{2}{\phi}\begin{pmatrix}
	\tilde{a}e_{\Omega} & \tilde{b}e_{\Omega} \\
	\tilde{a}f_{2\Omega} & \tilde{b}f_{2\Omega}
	\end{pmatrix} \mathbf{I}_{\Omega}$ is a skew-symmetric matrix on $U \setminus \lambda_{\Omega}^{-1}(0)$, but since $U \setminus \lambda_{\Omega}^{-1}(0)$ is dense, this is also true on $U$. So, there is $\omega_{1} \in C^{\infty}(U,\R)$ such that 
	\begin{align*}
	2\cD_{1}\mathbf{I}_{\Omega} - \mathbf{I}_{\Omega u_{1}} + \dfrac{2}{\phi}\begin{pmatrix}
	\tilde{a}e_{\Omega} & \tilde{b}e_{\Omega} \\
	\tilde{a}f_{2\Omega} & \tilde{b}f_{2\Omega}
	\end{pmatrix} \mathbf{I}_{\Omega} = \begin{pmatrix}
	0 & -\omega_{1}\\
	\omega_{1} & 0 
	\end{pmatrix}.
	\end{align*}
	Multiplying (\ref{equi2}) by the left side with $\begin{pmatrix}
	1 & 0
	\end{pmatrix}$ and by the right side with $\begin{pmatrix}
	0 & 1
	\end{pmatrix}^{T}$, we get
	\begin{align*}
	-\omega_1\lambda_{\Omega}=\bLambda_{(1)}^{\phantomsection}\begin{pmatrix}
	0&-\omega_1\\
	\omega_1&0
	\end{pmatrix}\bLambda_{(2)}^T=\bLambda_{(1)}^{\phantomsection}\mathbf{I}_\Omega\bLambda_{(2)u_1}^T-\bLambda_{(1)u_1}^{\phantomsection}\mathbf{I}_\Omega\bLambda_{(2)}^T-(E_{u_2}-F_{u_{1}}) \in \mathfrak{T}_{\Omega}.
	\end{align*}
	Reciprocally, suppose (\ref{cli1}). Since $U \setminus \lambda_{\Omega}^{-1}(0)$ is a dense set, there is a unique $\omega_{1} \in C^{\infty}(U,\R)$ such that
	\begin{align*}
	\bLambda_{(1)u_{1}}^{\phantomsection}\mathbf{I}_\Omega\bLambda_{(2)}^T-\bLambda_{(1)}^{\phantomsection}\mathbf{I}_\Omega\bLambda_{(2)u_{1}}^T+E_{u_{2}}-F_{u_{1}} = \omega_1\lambda_{\Omega} \in \mathfrak{T}_\Omega.
	\end{align*}
	From the above expression, it follows that $\bLambda\mathbf{I}_\Omega\bLambda_{u_1}^t-\bLambda_{u_1}\mathbf{I}_\Omega\bLambda^t+\mathbf{A}_1 = \bLambda \begin{pmatrix}
	0 & -\omega_{1}\\
	\omega_{1} & 0
	\end{pmatrix}\bLambda^{T}$, since $\bLambda\mathbf{I}_\Omega\bLambda_{u_1}^t-\bLambda_{u_1}\mathbf{I}_\Omega\bLambda^t+\mathbf{A}_1$ is a skew-symmetric matrix. Define $\cD_{1}: U \rightarrow M_{2\times 2}(\R)$, given by
	\begin{align*}
	\cD_{1} = \frac{1}{2}\left(\mathbf{I}_{\Omega u_{1}} - \dfrac{2}{\phi}\begin{pmatrix}
	\tilde{a}e_{\Omega} & \tilde{b}e_{\Omega} \\
	\tilde{a}f_{2\Omega} & \tilde{b}f_{2\Omega}
	\end{pmatrix} \mathbf{I}_{\Omega} +  \begin{pmatrix}
	0 & -\omega_1\\
	\omega_1 & 0 
	\end{pmatrix} \right)\mathbf{I}_{\Omega}^{-1} ,
	\end{align*}
	thus
	\begin{align*}
	\bLambda\left( 2\cD_{1}\mathbf{I}_{\Omega} - \mathbf{I}_{\Omega u_{1}} + \dfrac{2}{\phi}\begin{pmatrix}
	\tilde{a}e_{\Omega} & \tilde{b}e_{\Omega} \\
	\tilde{a}f_{2\Omega} & \tilde{b}f_{2\Omega}
	\end{pmatrix} \mathbf{I}_{\Omega} \right) \bLambda^{T}= \bLambda\mathbf{I}_\Omega\bLambda_{u_1}^t-\bLambda_{u_1}\mathbf{I}_\Omega\bLambda^t+\mathbf{A}_1.\label{equi22}
	\end{align*}
	Taking into account (\ref{eqdecI}) and then the expression for $\mathbf{I}_{u_{1}}$ we show, using the above expression, that $\cD_{1} = 	\bLambda^{-1}\left(\widetilde{\Gamma}_{1}\bLambda - \bLambda_{u_1} \right)$ on $U \setminus \lambda_{\Omega}^{-1}(0)$. Since $\cD_{1}$ is smooth and $U \setminus \lambda_{\Omega}^{-1}(0)$ is dense, we have the result.
\end{enumerate}
Item (b) follows analogously by considering the matrix 
$\mathbf{A}_2:=\begin{pmatrix}
0&-(F_{u_2}-G_{u_1})\\
F_{u_2}-G_{u_1}&0
\end{pmatrix}$.
\end{proof}

In preparation for the Fundamental theorem, let us set the matrices $\mathbf{W} = \begin{pmatrix}
\w_{1} & \w_{2} & \bxi
\end{pmatrix} \in GL(3)$,
\begin{subequations}\label{eqaux1}
	\begin{align}
	\mathbf{D}_{1} &= \begin{pmatrix}
	\cD^1_{11} & \cD^2_{11} & h_{11} \\
	\cD^1_{21} & \cD^2_{21} & h_{21}\\
	-S^{1}_{1}  & -S^{2}_{1}  & 0
	\end{pmatrix}\\
	\mathbf{D}_{2} &= \begin{pmatrix}
	\cD^1_{12} & \cD^2_{12} & h_{12} \\
	\cD^1_{22} & \cD^2_{22} & h_{22}\\
	-S^{1}_{2}  & -S^{2}_{2}  & 0
	\end{pmatrix}.
	\end{align}	
\end{subequations}
Then, the system (\ref{eqcompa}) is represented by
	\begin{align}\label{sistmat}
	\begin{cases}
	\mathbf{W}_{u_1} &= \mathbf{W}\mathbf{D}_{1}^{T}\\
	\mathbf{W}_{u_2} &= \mathbf{W}\mathbf{D}_{2}^{T}.
	\end{cases}
	\end{align}	
It is known that the compatibility condition for the system (\ref{sistmat}) is $\mathbf{W}_{u_{1}u_{2}}^{T} = \mathbf{W}_{u_{2}u_{1}}^{T}$, from which we obtain
\begin{align*}
{\bD}_{1}\bD_{2}\mathbf{W}^{T} + \bD_{1u_{2}}\mathbf{W}^{T}  = \bD_{1}\mathbf{W}_{u_2}^{T} + \bD_{1u_{2}}\mathbf{W}^{T} = \bD_{2}\mathbf{W}_{u_1}^{T} + \bD_{2u_{1}}\mathbf{W}^{T} = \bD_{2}\bD_{1}\mathbf{W}^{T} + \bD_{2u_{1}}\mathbf{W}^{T},
\end{align*}
that is equivalent to
\begin{align*}
\left(\bD_{1}\bD_{2} + \bD_{1u_{2}} - \bD_{2}\bD_{1} - \bD_{2u_{1}}\right)\mathbf{W}^{T} = \bm{0}.
\end{align*}
Since $\mathbf{W} \in GL(3)$, we get
\begin{align*}
\bD_{1u_{2}} - \bD_{2u_{1}} + \left[ \bD_{1}, \bD_{2} \right] = \bm{0},
\end{align*}
where $\left[ \bD_{1}, \bD_{2} \right] =  \bD_{1} \bD_{2} -  \bD_{2} \bD_{1}$.

We have next two auxiliary lemmas, which play an important role in the proof of theorem \ref{teofundamental}.
\begin{lema}\label{lemasce}
	The integrability conditions for the system
	\begin{align}\label{sistemafrontal0}
	\begin{cases}
	\bx_{u_1} = \lambda_{11} \w_{1} + \lambda_{12}\w_{2}\\
	\bx_{u_2} = \lambda_{21} \w_{1} + \lambda_{22}\w_{2}\\
	\bx(q) = p
	\end{cases}
	\end{align}
	are 
	\begin{itemize}
	\item $\bLambda \begin{pmatrix}
	h_{11} & h_{12}\\
	h_{21} & h_{22}
	\end{pmatrix}$ is symmetric; 
	\item $\begin{pmatrix}
	0 & 1
	\end{pmatrix}\left(\bLambda \cD_{1} + \bLambda_{u_1}\right) = \begin{pmatrix}
	1 & 0
	\end{pmatrix}\left(\bLambda \cD_{2} + \bLambda_{u_2} \right). $
\end{itemize}
\end{lema}

\begin{proof}
It is known that the integrability condition for the system (\ref{sistemafrontal0}) is $\bx_{u_1u_{2}} = \bx_{u_2u_{1}}$. If we set the matrices
\begin{align*}
\widetilde{\bLambda} = \begin{pmatrix}
\lambda_{11} & \lambda_{12} & 0\\
\lambda_{21} & \lambda_{22} & 0\\
0 & 0 & 1
\end{pmatrix}\; \text{and}\; \mathbf{M} = \begin{pmatrix}
\bx_{u_1} & \bx_{u_2} & \bxi
\end{pmatrix},
\end{align*}	
then $\mathbf{M} = \mathbf{W}\widetilde{\bLambda}^{T}$. Hence, the integrability condition is $\mathbf{M}_{u_1}\mathbf{\hat{j}} = \mathbf{M}_{u_2}\mathbf{\hat{i}}$, where $\lbrace \mathbf{\hat{i}}, \mathbf{\hat{j}}, \mathbf{\hat{k}} \rbrace$ is the standard basis of $\R^3$. By using (\ref{sistmat}), $\mathbf{M} = \mathbf{W}\widetilde{\bLambda}^{T}$ and $\mathbf{M}_{u_1}\mathbf{\hat{j}} = \mathbf{M}_{u_2}\mathbf{\hat{i}}$, we obtain
\begin{align*}
\bW\bD_{1}^T\widetilde{\bLambda}^{T}\mathbf{\hat{j}} + \bW\widetilde{\bLambda}_{u_{1}}^T\mathbf{\hat{j}}=\bW_{u_{1}}\widetilde{\bLambda}^T\mathbf{\hat{j}}+\bW\widetilde{\bLambda}_{u_1}^T\mathbf{\hat{j}}=\bW_{u_{2}}\widetilde{\bLambda}^T\mathbf{\hat{i}}+\bW\widetilde{\bLambda}_{u_{2}}^T\mathbf{\hat{i}}=\bW\bD_{2}^T\widetilde{\bLambda}^T\mathbf{\hat{i}}+\bW\widetilde{\bLambda}_{u_{2}}^T\mathbf{\hat{i}},
\end{align*} 
so, the integrability condition is equivalent to $\bD_{1}^T\widetilde{\bLambda}^T\mathbf{\hat{j}} + \widetilde{\bLambda}_{u_{1}}^T\mathbf{\hat{j}} = \bD_{2}^T\widetilde{\bLambda}^T\mathbf{\hat{i}}+\widetilde{\bLambda}_{u_{2}}^T\mathbf{\hat{i}}$. By taking each component of this expression we have
\begin{align}
	&\lambda_{11u_{2}}-\lambda_{21u_{1}}=\cD_{11}^{1}\lambda_{21}+\cD_{21}^{1}\lambda_{22}-\cD_{12}^{1}\lambda_{11}-\cD_{22}^{1}\lambda_{12}\label{CS1}\\
	&\lambda_{12u_{2}}-\lambda_{22u_{1}}=\cD_{11}^{2}\lambda_{21}+\cD_{21}^{2}\lambda_{22}-\cD_{12}^{2}\lambda_{11}-\cD_{22}^{2}\lambda_{12}\label{CS2}\\
	&\lambda_{11}h_{12}+\lambda_{12}h_{22}=\lambda_{21}h_{11}+\lambda_{22}h_{21}\label{CS3}.
\end{align}
Finally, note that (\ref{CS1}) and (\ref{CS2}) are equivalent to $\begin{pmatrix}
0 & 1
\end{pmatrix}\left(\bLambda \cD_{1} + \bLambda_{u_1}\right) = \begin{pmatrix}
1 & 0
\end{pmatrix}\left(\bLambda \cD_{2} + \bLambda_{u_2} \right)$ and (\ref{CS3}) is equivalent to say that $\bLambda \begin{pmatrix}
h_{11} & h_{12}\\
h_{21} & h_{22}
\end{pmatrix}$ is symmetric.
\end{proof}	

\begin{lema}{\rm{({\cite{alexandro2019fundamental}, Lemma 5.2})}}\label{lema5.2}
		If we have
		\begin{align*}
	\widetilde{\bLambda}	\mathbf{D}_{1} =	\mathbf{\widetilde{\Gamma}}_{1}\widetilde{\bLambda} - \widetilde{\bLambda}_{u_1}\; \text{and}\;		\widetilde{\bLambda}\mathbf{D}_{2} =	\mathbf{\widetilde{\Gamma}}_{2}\widetilde{\bLambda} - \widetilde{\bLambda}_{u_2},
	\end{align*}
	in which $\widetilde{\bLambda},\mathbf{D}_{1},\mathbf{D}_{2}:U \to M_{n\times n}(\R)$ and $\mathbf{\widetilde{\Gamma}}_{1},\mathbf{\widetilde{\Gamma}}_{2}:U\setminus \lambda_{\Omega}^{-1}(0) \to M_{n\times n}(\R)$ are smooth maps with $int(\lambda_{\Omega}^{-1}(0))=\emptyset$, where $\lambda_{\Omega} = \det \bLambda$. Then,\\
	
	$\mathbf{\widetilde{\Gamma}}_{1u_{2}} - \mathbf{\widetilde{\Gamma}}_{2u_{1}} + \left[ \mathbf{\widetilde{\Gamma}}_{1}, \mathbf{\widetilde{\Gamma}}_{2} \right] = \bm{0}$ is equivalent to $\bD_{1u_{2}} - \bD_{2u_{1}} + \left[ \bD_{1}, \bD_{2} \right] = \bm{0}$ on $U$. 
	\end{lema}

\begin{teo}\label{teofundamental}
	Let $\widetilde{\Gamma}^{k}_{ij}, b^{i}_{j},c_{ij},\phi \in C^{\infty}(U, \R)$, such that $\phi \neq 0$. Suppose that $\begin{pmatrix}
		c_{ij}
	\end{pmatrix} = \begin{pmatrix}
		\lambda_{ij}
	\end{pmatrix}  \begin{pmatrix}
		h_{ij} 
	\end{pmatrix},\;
\label{eqaux12}
\begin{pmatrix}
	S_{i}^{j}
\end{pmatrix} =  \begin{pmatrix}
	b_{i}^{j} 
\end{pmatrix} \begin{pmatrix}
	\lambda_{ij}
\end{pmatrix}$ and
	\begin{align}\label{eqaux11}
	\begin{pmatrix}
	E & F \\
	F & G
	\end{pmatrix}  &= \begin{pmatrix}
	\lambda_{11} & \lambda_{12}\\
	\lambda_{21} & \lambda_{22}
	\end{pmatrix} \begin{pmatrix}
	E_{\Omega} & F_{\Omega}\\
	F_{\Omega} & G_{\Omega}
	\end{pmatrix} \begin{pmatrix}
	\lambda_{11} & \lambda_{12}\\
	\lambda_{21} & \lambda_{22}
	\end{pmatrix}^{T}
		\end{align}
	where all the components above are $C^{\infty}$ functions defined on $U$, $E_{\Omega}, G_{\Omega} \geq 0$, $E_{\Omega}G_{\Omega} - F_{\Omega}^2 > 0$, $\lambda_{\Omega} = \det \bLambda$ and $U \setminus \lambda_{\Omega}^{-1}(0)$ is a dense set, for $\bLambda = \begin{pmatrix}
	\lambda_{ij}
	\end{pmatrix}$. Suppose also that the compatibility equations for the system (\ref{eqcompa0}) are satisfied on $U \setminus \lambda_{\Omega}^{-1}(0)$ and that \begin{subequations}\label{condext}
		\begin{align}
		\bLambda_{(1)u_{1}}^{\phantomsection}\mathbf{I}_\Omega\bLambda_{(2)}^T-\bLambda_{(1)}^{\phantomsection}\mathbf{I}_\Omega\bLambda_{(2)u_{1}}^T+E_{u_{2}}-F_{u_{1}} &\in \mathfrak{T}_\Omega\label{cli11}\\
		\bLambda_{(1)u_{2}}^{\phantomsection}\mathbf{I}_\Omega\bLambda_{(2)}^T-\bLambda_{(1)}^{\phantomsection}\mathbf{I}_\Omega\bLambda_{(2)u_{2}}^T+F_{u_{2}}-G_{u_{1}} &\in \mathfrak{T}_\Omega\label{cli22},
		\end{align}
	\end{subequations} where $\mathfrak{T}_\Omega$ is the principal ideal generated by $\lambda_{\Omega}$ in the ring $C^{\infty}(U, \R)$ and $\bI_{\Omega} = \begin{pmatrix}
E_{\Omega} & F_{\Omega}\\
F_{\Omega} & G_{\Omega}
\end{pmatrix}$.  Then,
	
	\begin{enumerate}[(a)]
		\item For each $q \in U$, there exists a neighborhood $V \subset U$ of $q$, a frontal $\bx: V \rightarrow \R^3$ with tmb $\bomega$, such that $D\bx = \bomega \bLambda^{T}$, and an equiaffine transversal vector field $\bxi: V \rightarrow \R^3$ with associated equiaffine structure given by $\begin{pmatrix} h_{ij}
		\end{pmatrix}, \cD_{1}, \cD_{2}$, where
		\begin{align*}
		\cD_{1} &= \begin{pmatrix}
		\cD^1_{11} & \cD^2_{11} \\
		\cD^1_{21} & \cD^2_{21}
		\end{pmatrix}\\
		\cD_{2} &= \begin{pmatrix}
		\cD^1_{12} & \cD^2_{12}\\
		\cD^1_{22} & \cD^2_{22}
		\end{pmatrix},
		\end{align*}
		are the unique $C^{\infty}$ extensions of  $	\bLambda^{-1}\left(\widetilde{\Gamma}_{1}\bLambda - \bLambda_{u_1} \right)$ and $	\bLambda^{-1}\left(\widetilde{\Gamma}_{2}\bLambda - \bLambda_{u_2} \right)$ to $U$, respectively. 
		
		\item If moreover, we suppose $\det \begin{pmatrix}
			c_{ij}
		\end{pmatrix} \neq 0$ on $V \setminus \lambda_{\Omega}^{-1}(0)$ and that the condition $\nabla\omega_{\mathbf{c}} = 0$ is satisfied on $V \setminus \lambda_{\Omega}^{-1}(0)$, where $\omega_{\mathbf{c}}$ is the volume element induced by the affine fundamental form $\mathbf{c}$, then there is a volume element $\omega$ in $\R^3$ such that $\bxi$ is the Blaschke vector field of the frontal $\bx$.
		
		\item Let $U$ be connected. Suppose that $\widetilde{\bx}: U \rightarrow \R^3$ is another proper frontal, $\widetilde{\bxi}$ an equiaffine transversal vector field and $\widetilde{\bomega}$ a tmb satisfying the same conditions that were obtained in (a). Then, $\bx$ and $\widetilde{\bx}$ are affinely equivalent.
	\end{enumerate}
\end{teo}

\begin{proof}
	\hfill \begin{enumerate}[(a)]
		\item It follows from (\ref{condext}) and from proposition (\ref{propextensao}) that the maps
		\begin{align*}
		\bLambda^{-1}\left(\widetilde{\Gamma}_{1}\bLambda - \bLambda_{u_1} \right)\; \text{and}\; \bLambda^{-1}\left(\widetilde{\Gamma}_{2}\bLambda - \bLambda_{u_2} \right),
				\end{align*}
		defined on $U \setminus \lambda_{\Omega}^{-1}(0)$, admit unique $C^{\infty}$ extensions to $U$, ${\cD}_{1}$ e ${\cD}_{2}$, respectively. Thus, on $U \setminus \lambda_{\Omega}^{-1}(0)$, we have
		\begin{align*}
		\cD_{1} &=	\bLambda^{-1}\left(\widetilde{\Gamma}_{1}\bLambda - \bLambda_{u_1} \right),\\
		\cD_{2} &=	\bLambda^{-1}\left(\widetilde{\Gamma}_{2}\bLambda - \bLambda_{u_2} \right).
		\end{align*}	
		By using $\cD_{1}$ and $\cD_{2}$, we can build the matrices $\mathbf{D}_{1}$ and $\mathbf{D}_{2}$ as the matrices (\ref{eqaux1}). Then, using (\ref{eqaux11}) and (\ref{eqaux12}) we have that
		\begin{align*}
	\widetilde{\bLambda}	\mathbf{D}_{1} =	\mathbf{\widetilde{\Gamma}}_{1}\widetilde{\bLambda} - \widetilde{\bLambda}_{u_1}\; \text{and}\;		\widetilde{\bLambda}\mathbf{D}_{2} =	\mathbf{\widetilde{\Gamma}}_{2}\widetilde{\bLambda} - \widetilde{\bLambda}_{u_2},
		\end{align*}	
		where 
		\begin{align*}
		\mathbf{\widetilde{\Gamma}}_{1} &= \begin{pmatrix}
		\widetilde{\Gamma}^{1}_{11} & \widetilde{\Gamma}^{2}_{11} & c_{11}\\
		\widetilde{\Gamma}^{1}_{21} & \widetilde{\Gamma}^{2}_{21} & c_{12}\\
		-b^{1}_{1} & -b^{2}_{1} & 0
		\end{pmatrix}
		\end{align*}
		\begin{align*}
		\mathbf{\widetilde{\Gamma}}_{2} &= \begin{pmatrix}
		\widetilde{\Gamma}^{1}_{12} & \widetilde{\Gamma}^{2}_{12} & c_{21}\\
		\widetilde{\Gamma}^{1}_{22} & \widetilde{\Gamma}^{2}_{22} &c_{22}\\
		-b_{2}^{1}  & -b^{2}_{2} & 0
		\end{pmatrix}\\
		\widetilde{\bLambda} &= \begin{pmatrix}
		\lambda_{11} & \lambda_{12} & 0 \\
		\lambda_{21} & \lambda_{22} & 0\\
		0 & 0 & 1
		\end{pmatrix}.
		\end{align*}
				Let us consider the following system of PDE
		\begin{subequations}\label{sistema54}
			\begin{align}
			\mathbf{W}_{u_1} &= \mathbf{W}\bD_{1}^{T}\\
			\mathbf{W}_{u_2} &= \mathbf{W}\bD_{2}^{T}\\
			\mathbf{W}(q) &= \begin{pmatrix}
			\mathbf{v_{1}} & \mathbf{v_{2}} & \mathbf{v_{3}}
			\end{pmatrix},
			\end{align}
			where $\mathbf{v_{1}}, \mathbf{v_{2}}, \mathbf{v_{3}}$ are linearly independent vectors on $\R^3$ and $q \in U$ is a fixed point.
		\end{subequations}	
		The compatibility conditions for the system (\ref{eqcompa0}) are expressed by $\mathbf{\widetilde{\Gamma}}_{1u_{2}} - \mathbf{\widetilde{\Gamma}}_{2u_{1}} + \left[ \mathbf{\widetilde{\Gamma}}_{1}, \mathbf{\widetilde{\Gamma}}_{2} \right] = \bm{0}$ and by hypothesis they are satisfied, so it follows from lemma \ref{lema5.2} that $\bD_{1u_{2}} - \bD_{2u_{1}} + \left[ \bD_{1}, \bD_{2} \right] = \bm{0}$, which is equivalent to the compatibility conditions for the system (\ref{sistema54}) (see \ref{sistmat}). Thus, this system has a unique solution $\mathbf{W}: V \rightarrow GL(3)$, where $V \subset U$ is a neighborhood of $q$. If $\mathbf{W} = \begin{pmatrix}
		\w_{1} & \w_{2} & \bxi
		\end{pmatrix}$, it follows that the vector field $\bxi$ is transversal to $\langle \w_{1}, \w_{2} \rangle_{\R}$ and that $\bxi_{u_{i}} \in \langle \w_{1}, \w_{2} \rangle_{\R}$, $i=1,2$. Now, let us take the following system of PDE
		\begin{align}\label{sistemafrontal}
			\begin{cases}
			\bx_{u_1} = \lambda_{11} \w_{1} + \lambda_{12}\w_{2}\\
			\bx_{u_2} = \lambda_{21} \w_{1} + \lambda_{22}\w_{2}\\
			\bx(q) = p,
			\end{cases}
		\end{align}
		for a fixed $p \in \R^3$. Note that 
		\begin{align*}
		\begin{pmatrix}
		0 & 1
		\end{pmatrix}\left(\bLambda \cD_{1} + \bLambda_{u_1}\right) = \begin{pmatrix}
		0 & 1
		\end{pmatrix} \widetilde{\Gamma}_{1} \bLambda = \begin{pmatrix}
		1 & 0
		\end{pmatrix}\widetilde{\Gamma}_{2} \bLambda  = \begin{pmatrix}
		1 & 0
		\end{pmatrix}\left(\bLambda \cD_{2} + \bLambda_{u_2} \right)
		\end{align*}
		on $\lambda_{\Omega}^{-1}(0)^{c}$, so, via density the equality holds on $U$. The above equality, together with the fact that $\bLambda \begin{pmatrix}
		h_{11} & h_{12}\\
		h_{21} & h_{22}
		\end{pmatrix}$ is symmetric, means that the system (\ref{sistemafrontal}) has a unique solution $\bx: \widetilde{V} \rightarrow \R^3$, where $ \widetilde{V} \subset V$ is a neighborhood of $q$ (see lemma \ref{lemasce}). As $D\bx = \bomega \bLambda^{T}$, where $\bomega = \begin{pmatrix}
		\w_{1} & \w_{2}
		\end{pmatrix}$ it follows that $\bx$ is a frontal for which $\bomega$ is a tmb, $\bxi$ is an equiaffine transversal vector field and the equiaffine structure is the desired one.
		
		\item Moreover, from $\det\begin{pmatrix}{c_ij} \end{pmatrix} \neq 0$ on $V \setminus \lambda_{\Omega}^{-1}(0)$ we obtain that $\mathbf{c}$ is non-degenerate on $V \setminus \lambda_{\Omega}^{-1}(0)$, thus let $\omega_{\mathbf{c}}$ be the volume element induced by the non-degenerate metric $\mathbf{c}$. Let $\omega_{1}$ be the volume element in $\R^3$, so the volume element induced by the equiaffine vector field $\bxi$ is $\theta_{1}(\bv_{1}, \bv_{2}) = \omega_{1}(\bv_{1}, \bv_{2}, \bxi)$, where $\bv_{1}, \bv_{2}$ are tangent vectors. Since $\bxi$ is equiaffine, $\theta_{1}$ is a parallel volume element in the regular part of $\bx$, i.e., $\nabla\theta_{1} = 0$ (see proposition \ref{prop1.4}). The apolarity condition $\nabla \omega_{\mathbf{c}} = 0$ means that $\omega_{\mathbf{c}} = \mu \theta_{1}$, where $\mu$ is a positive constant, since a parallel volume element is unique up to a positive scalar multiple. If we take in $\R^3$ the volume element $\omega = \mu \omega_{1}$ and the new induced volume element $\theta$, we get $\omega_{\mathbf{c}} = \theta$, therefore, $\bxi$ is the usual Blaschke vector field of $\bx$ in the regular part. Since $\bxi$ is defined on $V$, it follows that $\bxi$ is the Blaschke vector field of $\bx$ as defined in \ref{defBlaschke}.

		\item If we write  $\bomega = \begin{pmatrix}
		\w_{1} & \w_{2}
		\end{pmatrix}$ and $\widetilde{\bomega} = \begin{pmatrix}
		\widetilde{\w_{1}} & \widetilde{\w_{2}}
		\end{pmatrix}$, then for each $q \in U$ there is an isomorphism $L_{q}: \R^3 \rightarrow \R^3$, such that
		\begin{align*}
		L_{q}(\w_{i}) &= \widetilde{\w_{i}},\; i=1,2.\\
		L_{q}(\bxi) &= \widetilde{\bxi}.
		\end{align*}
		We seek to show that $L$ is constant. It follows from the fact that $\bx$ e $\widetilde{\bx}$ satisfy the same hypothesis given in item a) that both relative shape operators $S_{\Omega}$ and $\widetilde{S}_{\widetilde{\Omega}}$ are given by the matrix $\begin{pmatrix}
		S_{i}^{j}
		\end{pmatrix}$, hence $L_{q}(S(\w_{i})) = \widetilde{S}(\widetilde{\w_{i}})$. Thus,
		\begin{align*}
		-\widetilde{S}(\widetilde{\w_{i}}) = \dfrac{\partial}{\partial u_{i}}\widetilde{\bxi}
		=  \dfrac{\partial}{\partial u_{i}}L(\bxi) = \left(\dfrac{\partial}{\partial u_{i}}L\right)(\bxi) + L\left(\dfrac{\partial}{\partial u_{i}}\bxi\right)
		&= \left(\dfrac{\partial}{\partial u_{i}}L\right)(\bxi) - L(S(\w_{i}))\\
		&= \left(\dfrac{\partial}{\partial u_{i}}L\right)(\bxi) - \widetilde{S}(\widetilde{\w_{i}}),
		\end{align*} 
which means that $ \left(\dfrac{\partial}{\partial u_{i}}L\right)(\bxi) = 0$, $i=1,2$. Furthermore,
		\begin{align*}
		\widetilde{\w_{1}}_{u_{1}} = \dfrac{\partial}{\partial u_{1}}L(\w_{1})
		&= \left(\dfrac{\partial}{\partial u_{1}}L\right)(\w_{1}) + L(\w_{1u_1})\\
		&= \left(\dfrac{\partial}{\partial u_{1}}L\right)(\w_{1}) + L(\cD^{1}_{11} \w_{1} + \cD^{2}_{11}  \w_{2} + 
		h_{11}\bxi)\\
		&=  \left(\dfrac{\partial}{\partial u_{1}}L\right)(\w_{1}) + \widetilde{\w_{1}}_{u_{1}}, 
		\end{align*}
		so, $\left(\dfrac{\partial}{\partial u_{1}}L\right)(\w_{1}) = 0$. Analogously, we show that $\left(\dfrac{\partial}{\partial u_{i}}L\right)(\w_{j}) = 0$, $i,j=1,2$. Therefore, $L$ is constant and from $\widetilde{\bomega} = L \bomega$ we obtain
		\begin{align*}
		D\widetilde{\bx} = \widetilde{\bomega} \bLambda^{T} = L \bomega \bLambda^{T} = LD\bx, 
		\end{align*}
		that is, $\widetilde{\bx} = L\bx + \mathbf{a}$, where $\mathbf{a}$ is a constant vector and $L: \R^3 \rightarrow \R^3$ is a linear isomorphism.

	\end{enumerate}
\end{proof}	

\begin{obs}\normalfont
The approach used to prove item (a) in theorem \ref{teofundamental} is the same as that applied to prove the existence part of the fundamental theorem presented in \cite{alexandro2019fundamental}. Then, as we are considering here any equiaffine transversal vector field, it is possible to recover from item (a) the existence theorem presented in \cite{alexandro2019fundamental}, taking the unit normal as the equiaffine vector field.
\end{obs}

{\footnotesize \bibliographystyle{siam}
	\bibliography{Equiaffine_Structure_for_frontals}}
\end{document}